\documentclass[12pt]{article}
\usepackage{amsmath, amssymb,latexsym}
\usepackage{color}
\usepackage{graphicx}
\title{Eigenvalues and subelliptic estimates for non-selfadjoint semiclassical operators with double characteristics}
\author{Michael Hitrik\\\small Department of Mathematics
\\\small University of California \\\small Los Angeles CA \\\small 90095-1555, USA
\\\small hitrik@math.ucla.edu \and Karel Pravda-Starov
\\\small D\'epartement de Math\'ematiques
\\\small Universit\'e de Cergy-Pontoise
\\\small Site de St Martin, 2 avenue Adolphe Chauvin
\\\small 95302 Cergy-Pontoise Cedex, France
\\\small karel.pravda-starov@u-cergy.fr}
\date{}

\def\wrtext#1{\relax\ifmmode{\leavevmode\hbox{#1}}\else{#1}\fi}
\def\abs#1{\left|#1\right|}
\def\begeq{\begin{equation}}
\def\endeq{\end{equation}}

\def\neigh{neighborhood}

\def\Re{{\rm Re\,}}
\def\Im{{\rm Im\,}}

\newcommand{\eps}{\varepsilon}
\def\part#1{\frac{\partial}{\partial #1}}

\def\norm#1{||\,#1\,||}

\newcommand{\real}{\mbox{\bf R}}
\newcommand{\comp}{\mbox{\bf C}}

\newcommand{\nat}{\mbox{\bf N}}

\renewcommand{\Re}{\mbox{\rm Re\,}}
\renewcommand{\Im}{\mbox{\rm Im\,}}

\renewcommand{\exp}{\mbox{\rm exp\,}}

\newtheorem{dref}{Definition}[section]
\newtheorem{lemma}[dref]{Lemma}
\newtheorem{theo}[dref]{Theorem}
\newtheorem{prop}[dref]{Proposition}

\newenvironment{proof}{\vspace{.3cm}\noindent{{\em Proof:}}}{\hfill$\Box$}

\begin{document}
\maketitle

\vspace*{1cm}
\noindent
{\bf Abstract}: For a class of non-selfadjoint $h$--pseudodifferential operators with double characteristics, we give a precise description
of the spectrum and establish accurate semiclassical resolvent estimates in a neighborhood of the origin. Specifically,
assuming that the quadratic approximations of the principal symbol of the operator along the double characteristics enjoy a partial ellipticity
property along a suitable subspace of the phase space, namely their singular space, we give a precise description of the spectrum of the operator in
an ${\cal O}(h)$--neighborhood of the origin. Moreover, when all the singular spaces are reduced to zero, we establish accurate semiclassical
resolvent estimates of subelliptic type, which depend directly on algebraic properties of the Hamilton maps associated to the quadratic
approximations of the principal symbol.


\tableofcontents
\section{Introduction}
\setcounter{equation}{0}

In this work, we are concerned with the analysis of spectral properties for general non-selfadjoint pseudodifferential operators with double characteristics. This study was initiated in~\cite{HiPr2}, and our purpose here is to complement the results of~\cite{HiPr2} on two essential points, as we describe below.
Assume that we are given a non-selfadjoint semiclassical pseudodifferential operator
$$
P=P^w(x,hD_x;h), \ 0<h \leq 1;
$$
defined by the semiclassical Weyl quantization of the symbol $P(x,\xi;h)$,
$$
P^w(x,hD_x;h)u(x)=\frac{1}{(2\pi)^n}\int_{{\bf R}^{2n}}e^{i(x-y).\xi}P\Big(\frac{x+y}{2},h\xi;h\Big)u(y)dyd\xi,
$$
with a semiclassical asymptotic expansion
$$
P(x,\xi;h) \sim \sum_{j=0}^{+\infty}h^j p_j(x,\xi),
$$
such that its principal symbol $p_0$ has a non-negative real part
$$
\textrm{Re }p_0(X)\geq 0,\ X=(x,\xi) \in \real^{2n},
$$
and such that we have a finite number of doubly characteristic points $X_0$ for the operator,
$$
p_0(X_0)=\nabla p_0(X_0)=0.
$$
Our interest is in studying spectral properties and the resolvent growth of the operator $P$ in a fixed neighborhood of the origin. In the previous work~\cite{HiPr2}, we established an accurate semiclassical a priori estimate
\begin{equation}
\label{dl2}
h \norm{u}_{L^2} \leq C_0\norm{(P-hz)u}_{L^2},\,\,\, \abs{z} \leq C,
\end{equation}
valid in an ${\cal O}(h)$-neighborhood of the origin, when the quadratic approximations $q$ of the principal symbol $p_0$ at the doubly characteristic points enjoy the partial ellipticity property
\begeq
\label{ell}
(x,\xi) \in S, \ q(x,\xi)=0\,  \Rightarrow\,  (x,\xi)=0.
\endeq
Here $S$ is a suitable subspace of the phase space, namely the singular space associated to $q$~\cite{HiPr1}, and the spectral parameter $z$ in (\ref{dl2}) avoids a discrete set depending on the values of the subprincipal symbol $p_1$ and the spectra of the quadratic approximations of the
principal symbol $p_0$ at the doubly characteristic points. The a priori estimate (\ref{dl2}) gives a first localization and bounds on the low lying eigenvalues of the operator $P$, i.e., when restricting the attention to an ${\cal O}(h)$-neighborhood of the origin in the complex spectral plane. In the first part of the present work, we shall push this analysis further and give a precise description of the spectrum of the operator $P$ in an ${\cal O}(h)$-neighborhood of the origin, with complete semiclassical asymptotic expansions for the eigenvalues. That such a study is planned by the authors was mentioned in~\cite{HiPr2}.

\medskip
\noindent
In the second part of this work, we shall be concerned with the behavior of the resolvent norm of $P$ in a sufficiently small but fixed neighborhood of the origin. We shall actually show that this behavior is linked to subelliptic properties of the quadratic approximations of the principal symbol $p_0$ at the doubly characteristic points, and that the positive integers $k_0$ appearing in the resolvent estimates
$$
h^{\frac{2k_0}{2k_0+1}}|z|^{\frac{1}{2k_0+1}}\|u\|_{L^2} \leq C_0\|Pu-zu\|_{L^2},
$$
depend directly on the loss of derivatives associated to the subelliptic properties of these quadratic operators. We shall show how the positive integers $k_0$ are intrinsically associated to the structure of the doubly characteristic set, and how they are completely characterized by algebraic properties of the Hamilton maps associated to the quadratic approximations of the principal symbol.

\medskip
\noindent
As in \cite{HiPr2}, the starting point for this work has been the general study of the Kramers-Fokker-Planck type operators carried out by
F.~H\'erau, J.~Sj\"ostrand and C.~Stolk in~\cite{HeSjSt}. This study has been a major breakthrough in the understanding of the spectral
properties of some general classes of pseudodifferential operators that are neither selfadjoint nor elliptic. We draw our
inspiration considerably from this work and use many techniques developed in the analysis of~\cite{HeSjSt}. By using some of these techniques,
together with the recent improvements in the understanding of spectral and subelliptic properties of non-elliptic quadratic operators obtained
in~\cite{HiPr1} and~\cite{Karel11}, here we are able to extend to a large class of non-selfadjoint semiclassical pseudodifferential operators
with double characteristics the results proved in~\cite{HeSjSt} for the case of operators of Kramers-Fokker-Planck type.

\subsection{Miscellaneous facts about quadratic differential operators}
Before giving the precise statement of the main results contained in this article, we shall recall miscellaneous facts and notation concerning
quadratic differential operators. Associated to a complex-valued quadratic form
\begin{eqnarray*}
q : \real_x^n \times \real_{\xi}^n &\rightarrow& \comp \\
 (x,\xi) & \mapsto & q(x,\xi),
\end{eqnarray*}
with $n \in \nat^*$, is the Hamilton map $F \in M_{2n}(\comp)$ uniquely defined by the identity
\begin{equation}
\label{10}
q\big{(}(x,\xi);(y,\eta) \big{)}=\sigma \big{(}(x,\xi),F(y,\eta) \big{)}, \ (x,\xi) \in \real^{2n},  (y,\eta) \in \real^{2n},
\end{equation}
where $q\big{(}\textrm{\textperiodcentered};\textrm{\textperiodcentered} \big{)}$ stands for the polarized form
associated to the quadratic form $q$ and $\sigma$ is the canonical symplectic form on $\real^{2n}$,
\begin{equation}
\label{11}
\sigma \big{(}(x,\xi),(y,\eta) \big{)}=\xi.y-x.\eta, \ (x,\xi) \in \real^{2n},  (y,\eta) \in \real^{2n}.
\end{equation}
It follows directly from the definition of the Hamilton map $F$ that
its real and imaginary parts, denoted respectively by $\textrm{Re } F$ and $\textrm{Im }F$,
$$\textrm{Re }F=\frac{1}{2}(F+\overline{F}), \ \textrm{Im }F=\frac{1}{2i}(F-\overline{F}),$$
with $\overline{F}$ being the complex conjugate of $F$, are the Hamilton maps associated
to the quadratic forms $\textrm{Re } q$ and $\textrm{Im }q$, respectively; and that a
Hamilton map is always skew-symmetric with respect to $\sigma$. This fact is just a consequence of the
properties of the skew-symmetry of the symplectic form and the symmetry of the polarized form,
\begin{equation}
\label{12}
\forall X,Y \in \real^{2n},\, \sigma(X,FY)=q(X;Y)=q(Y;X)=\sigma(Y,FX)=-\sigma(FX,Y).
\end{equation}
We defined in~\cite{HiPr1} the singular space $S$ associated to the quadratic symbol $q$ as the following intersection of kernels,
\begin{equation}\label{h1}
S=\Big(\bigcap_{j=0}^{2n-1}\textrm{Ker}\big[\textrm{Re }F(\textrm{Im }F)^j \big]\Big) \bigcap \real^{2n},
\end{equation}
where $F$ stands for the Hamilton map of $q$, and we proved in Theorem~1.2.2 in~\cite{HiPr1}, that when a quadratic symbol $q$ with a non-negative real part is elliptic on its singular space $S$,
\begin{equation}
\label{sm2}
(x,\xi) \in S, \ q(x,\xi)=0\, \Rightarrow\,  (x,\xi)=0,
\end{equation}
then the spectrum of the quadratic operator  $q^w(x,D_x)$ is only composed of eigenvalues of finite multiplicity and is given by
\begin{equation}
\label{sm6}
\sigma\big{(}q^w(x,D_x)\big{)}=\Big\{ \sum_{\substack{\lambda \in \sigma(F), \\  -i \lambda \in {\bf C}_+
\cup (\Sigma(q|_S) \setminus \{0\})
} }
{\big{(}r_{\lambda}+2 k_{\lambda}
\big{)}(-i\lambda) : k_{\lambda} \in \nat}
\Big\}.
\end{equation}
Here $r_{\lambda}$ is the dimension of the space of generalized eigenvectors of $F$ in $\comp^{2n}$ belonging to the eigenvalue $\lambda \in \comp$, and
$$
\Sigma(q|_S)=\overline{q(S)} \textrm{ and } \comp_+=\{z \in \comp : \textrm{Re }z>0\}.
$$
It follows from (\ref{h1}) that the closure of the range of $q$ along $S$, $\Sigma(q|_S)$, satisfies $\Sigma(q|_S) \subset i\real$.

\bigskip
\noindent
\textit{Remark.} Equivalently, one can describe the singular space as the subset in the phase space where all the Poisson brackets
$H_{\textrm{Im}\, q}^k\textrm{Re }q$, $k \in \nat$, are vanishing,
$$S=\{X \in \real^{2n} : H_{\textrm{Im}\, q}^k\textrm{Re }q(X)=0, k \in \nat\}.$$
The singular space is therefore exactly the set of points $X_0$ in the phase space where the real part of $q$ under the flow generated by the
Hamilton vector field associated to its imaginary part $\textrm{Im }q$,
$$t \mapsto \textrm{Re }q(e^{tH_{\textrm{Im}\,q}}X_0),$$
vanishes to an infinite order at $t=0$. We refer to Section~2 in~\cite{HiPr1} to find all the arguments needed to establish this second
equivalent description of the singular space.

\bigskip
\noindent
We shall finish this subsection by recalling that quadratic operators with a zero singular space $S=\{0\}$, enjoy noticeable subelliptic properties.
Specifically, when $q^w(x,D_x)$ stands for a quadratic operator whose Weyl symbol $q$ has a non-negative real part
$\textrm{Re }q \geq 0$, and a zero singular space $S=\{0\}$, it was established in~\cite{Karel11} that it fulfills the subelliptic estimate
\begin{equation}
\label{dl1}
\big\|\big(\langle(x,\xi)\rangle^{2/(2k_0+1)}\big)^w u\big\|_{L^2} \leq C\big(\|q^w(x,D_x) u\|_{L^2}+\|u\|_{L^2}\big), \ u \in \mathcal{S}(\real^n),
\end{equation}
with a loss of $2k_0/(2k_0+1)$ derivatives, where $\langle(x,\xi)\rangle=(1+|x|^2+|\xi|^2)^{1/2}$ and $k_0$ stands for the smallest integer
$0 \leq k_0 \leq 2n-1$ such that
$$
\Big(\bigcap_{j=0}^{k_0}\textrm{Ker}\big[\textrm{Re }F(\textrm{Im }F)^j \big]\Big) \bigcap \real^{2n}=\{0\}.
$$
Such a non-negative integer $k_0$ is well-defined since $S=\{0\}$.

\subsection{Statement of the main results}
Let us now state the main results contained in this paper. Let $m\geq 1$  be a $C^{\infty}$ order function on $\real^{2n}$ fulfilling
\begin{equation}
\label{eq1.1}
\exists C_0 \geq 1, N_0>0,\ m(X)\leq C_0 \langle{X-Y\rangle}^{N_0} m(Y),\ X,Y\in \real^{2n},
\end{equation}
where $\langle X \rangle=(1+|X|^2)^{\frac{1}{2}}$, and let $S(m)$ be the symbol class
$$
S(m)=\left\{ a\in C^{\infty}(\real^{2n},\comp): \forall \alpha \in \nat^{2n}, \exists C_{\alpha}>0, \forall X \in \real^{2n},
\  |\partial_X^{\alpha} a(X)| \leq C_{\alpha} m(X)\right\}.$$
We shall assume in the following, as we may, that $m$ belongs to its own symbol class $m\in S(m)$.

\medskip
\noindent
Considering a symbol $P(x,\xi;h)$ with a semiclassical asymptotic expansion in the symbol class $S(m)$,
\begin{equation}
\label{xi1}
P(x,\xi;h) \sim \sum_{j=0}^{+\infty} h^j p_j(x,\xi),
\end{equation}
with some $p_j \in S(m)$, $j \in \nat$, independent of the semiclassical parameter $h$, such that its principal symbol $p_0$ has a non-negative real part
\begin{equation}
\label{eq1.4}
\textrm{Re }p_0(X)\geq 0,\ X=(x,\xi) \in \real^{2n},
\end{equation}
we shall study the operator
\begin{equation}
\label{eq1.3}
P = P^w(x,hD_x;h),\ 0<h\leq 1,
\end{equation}
defined by the $h$-Weyl quantization of the symbol $P(x,\xi;h)$, that is, the Weyl quantization of the symbol $P(x,h\xi;h)$,
\begin{equation}\label{quant}
P^w(x,hD_x;h)u(x)=\frac{1}{(2\pi)^n}\int_{{\bf R}^{2n}}e^{i(x-y).\xi}P\Big(\frac{x+y}{2},h\xi;h\Big)u(y)dyd\xi.
\end{equation}

\medskip
\noindent
We shall make the important assumption that $\Re p_0$ is elliptic at infinity in the sense that for some
$C>1$, we have
\begeq
\label{eq1.5}
\Re p_0(X) \geq \frac{m(X)}{C},\quad \abs{X}\geq C.
\endeq
The ellipticity assumption (\ref{eq1.5}) implies that, for $h>0$ small enough and when equipped with the domain
$$
{\cal D}(P) = H(m):=\left(m^w(x,hD)\right)^{-1}\left(L^2(\real^n)\right),
$$
the operator $P$ becomes closed and densely defined on $L^2(\real^n)$. Furthermore, another basic consequence of (\ref{eq1.4}) and (\ref{eq1.5}) is
that when $z\in {\rm neigh}(0,\comp)$, the analytic family of operators
$$
P-z: H(m)\rightarrow L^2(\real^n),
$$
is Fredholm of index $0$, for all $h>0$ small enough --- see, e.g., \cite{DeSjZw}. An application of analytic Fredholm theory allows us then to
conclude that the spectrum of $P$ in a small  but fixed \neigh{} of $0\in \comp$ is discrete and consists of eigenvalues of finite algebraic
multiplicity.

\medskip
We shall assume that the characteristic set of the real part of the principal symbol~$p_0$,
$$(\textrm{Re }p_0)^{-1}(0)\subset \real^{2n},$$
is finite, so that we may write it as
\begin{equation}
\label{eq1.6}
(\textrm{Re } p_0)^{-1}(0) = \{X_1,...,X_N\}.
\end{equation}
The sign assumption (\ref{eq1.4}) implies in particular that we have
$$
d\textrm{Re }p_0(X_j)=0,
$$
for all $1 \leq j \leq N$, and we shall actually assume that these points are all doubly characteristic for the full principal symbol $p_0$,
\begin{equation}
\label{eq1.6.5}
p_0(X_j)=dp_0(X_j)=0,\ 1\leq j\leq N,
\end{equation}
so that we may write
\begin{equation}
\label{eq1.6.6}
p_0(X_j+Y)=q_j(Y)+\mathcal{O}(Y^3),
\end{equation}
when $Y\rightarrow 0$. Here $q_j$ is the quadratic form which begins the Taylor expansion of the principal symbol $p_0$ at $X_j$.
Notice that the sign assumption (\ref{eq1.4}) implies that the complex-valued quadratic forms $q_j$ have non-negative real parts,
\begin{equation}
\label{kps1}
\Re q_j\geq 0,
\end{equation}
when $1 \leq j \leq N$. We shall assume throughout the present work that when $1\leq j \leq N$, the quadratic form $q_j$ is elliptic along the associated singular space $S_j$ introduced in (\ref{h1}), in the sense of (\ref{ell}).

\bigskip
\noindent
The following result was established in~\cite{HiPr2}, under the assumptions above: let $C>1$ and assume that $z\in \comp$ with $\abs{z}\leq C$ is such that for all $1\leq j \leq N$, we have $z - p_1(X_j) \notin \Omega_j$, where $\Omega_j\subset \comp$ is a fixed \neigh{} of the spectrum of the quadratic operator $q_j^w(x,D_x)$. Then for all $h>0$ small enough, the following a priori estimate holds,
\begeq
\label{eq1.6.7}
h\norm{u}\leq {\cal O}(1) \norm{(P-hz)u},\quad u\in {\cal S}(\real^n).
\endeq
Here $\norm{\cdot}$ is the $L^2$--norm on $\real^n$. In view of the observations made above, we see that the estimate (\ref{eq1.6.7}) extends to all of ${\cal D}(P)=H(m)$, since the Schwartz space ${\cal S}(\real^n)$ is dense in the latter. The operator $P-hz: H(m)\rightarrow L^2(\real^n)$ is therefore injective with closed range, and thus invertible, thanks to the Fredholm property. We conclude that when $z\in \comp$ is as above, then $hz$ is not an
eigenvalue of $P$ and the resolvent estimate
\begeq
\label{eq1.6.8}
\left(P-hz\right)^{-1}  = {\cal O}\left(\frac{1}{h}\right): L^2(\real^n) \rightarrow L^2(\real^n)
\endeq
holds true.

\bigskip
\noindent
The following is the first main result of this work.
\begin{theo}
\label{theo1}
Let us make the assumptions {\rm (\ref{eq1.4})}, {\rm (\ref{eq1.5})}, {\rm (\ref{eq1.6})}, and {\rm (\ref{eq1.6.5})}. Assume
furthermore that the quadratic form $q_j$ introduced in {\rm (\ref{eq1.6.6})} is elliptic along the singular space $S_j$, when $1\leq j \leq N$.
Let $C>0$. Then there exists $h_0>0$ such that for all $0< h \leq h_0$, the spectrum of the operator $P$ in the open disc in the complex plane $D(0,Ch)$ is given by the eigenvalues of the form,
\begeq
\label{eq1.9}
z_{j,k} \sim h \left(\lambda_{j,k} + p_1(X_j) + h^{1/N_{j,k}} \lambda_{j,k,1} + h^{2/N_{j,k}} \lambda_{j,k,2} +\ldots\right),\,\, 1\leq j \leq N.
\endeq
Here $\lambda_{j,k}$ are the eigenvalues in $D(0,C)$ of $q_j^w(x,D_x)$ given in {\rm (\ref{sm6})}, repeated according to their algebraic multiplicity,  and $N_{j,k}$ is the dimension of the corresponding generalized eigenspace. (Possibly after changing $C>0$, we may assume that
$\abs{\lambda_{j,k}+p_1(X_j)}\neq C$ for all $k$, $1 \leq j \leq N$.)
\end{theo}

\bigskip
\noindent
We now come to state the second main result of this work. In doing so, let us introduce the symbols
\begin{equation}
\label{hel1}
r_j(Y)=p_0(X_j+Y)-q_j(Y), \,\, 1 \leq j \leq N.
\end{equation}
We shall assume that there exists a closed angular sector $\Gamma$ with vertex at 0 and a neighborhood $V$ of the origin in $\real^{2n}$
such that for all $1 \leq j \leq N$,
\begin{equation}\label{re1}
r_j(V) \setminus \{0\} \subset \Gamma \setminus \{0\} \subset \{z \in \comp: \textrm{Re }z > 0\}.
\end{equation}

\begin{figure}[hhh]
\label{des1}
\caption{The range of $r_j$.}
\includegraphics[scale=0.7]{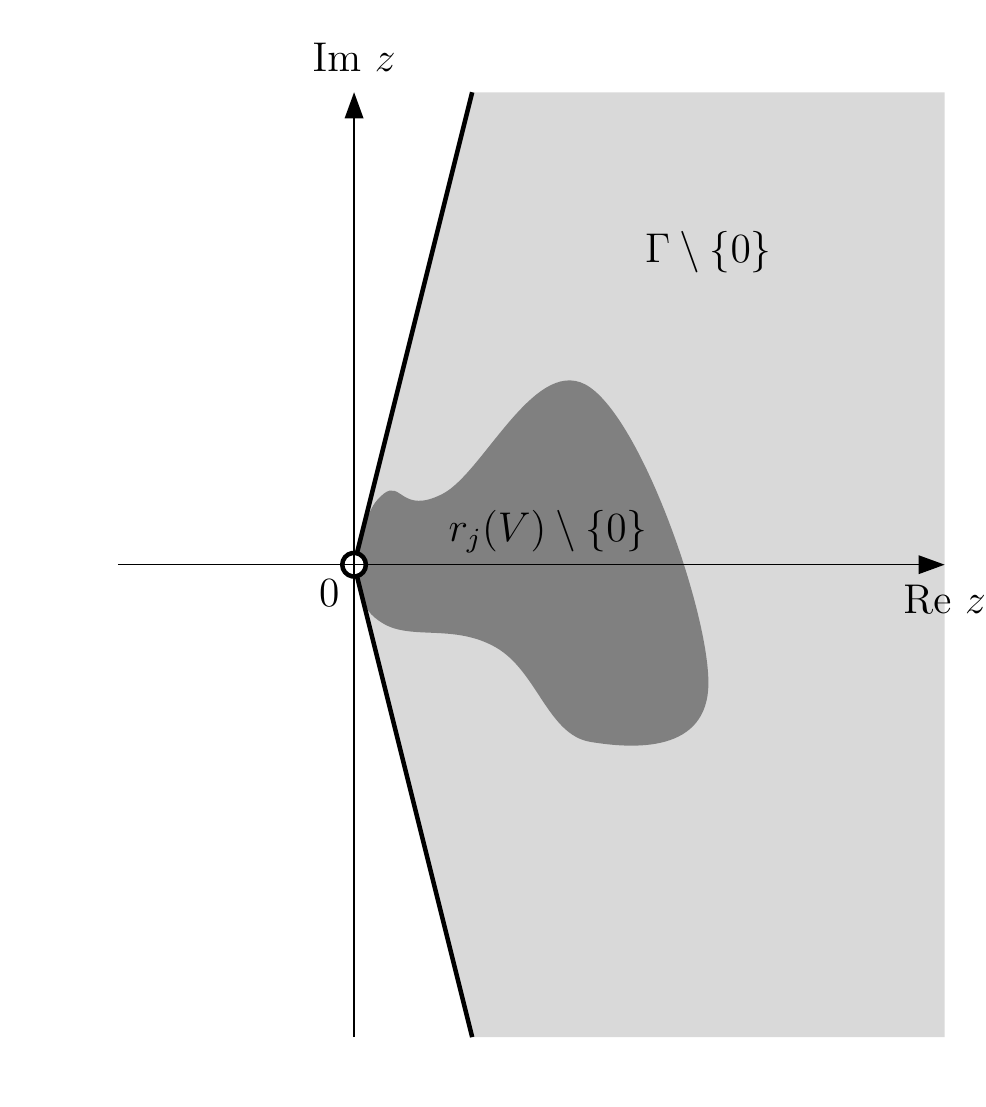}
\end{figure}
By denoting $F_j$ the Hamilton maps and $S_j$ the singular spaces associated to the quadratic forms $q_j$, we shall also assume that all the
singular spaces are reduced to zero,
\begin{equation}\label{ning2}
S_j=\{0\},
\end{equation}
when $1 \leq j \leq N$. According to the definition of the singular space (\ref{h1}), one can therefore consider the smallest integers,
$0 \leq k_j \leq 2n-1$, such that
\begin{equation}\label{ning1}
\Big(\bigcap_{l=0}^{k_j}\textrm{Ker}\big[\textrm{Re }F_j(\textrm{Im }F_j)^l \big]\Big) \bigcap \real^{2n}=\{0\}.
\end{equation}
Defining the integer
\begin{equation}\label{ning3}
k_0= \max_{j=1,...,N} k_j,
\end{equation}
in $\{0,...,2n-1\}$, we shall establish the following result:

\begin{theo}
\label{theo}
Consider a symbol $P(x,\xi;h)$ with a semiclassical expansion in the class $S(m)$ fulfilling the assumptions {\rm (\ref{eq1.4})},
{\rm (\ref{eq1.5})}, {\rm (\ref{eq1.6})}, {\rm (\ref{eq1.6.5})} and
{\rm (\ref{re1})}. When all the quadratic forms $q_j$, $1 \leq j \leq N$, defined in {\rm (\ref{eq1.6.6})} have zero singular spaces $S_j=\{0\}$,
then for any constant $C_0 > 0$ sufficiently small, there exist positive constants $0<h_0 \leq 1$, $C \geq 1$ and $c_0>0$ such that for all
$0<h \leq h_0$, $u \in \mathcal{S}(\real^n)$ and $z \in \Omega_{h}$,
\begin{equation}
\label{ning4}
 h^{\frac{2k_0}{2k_0+1}}|z|^{\frac{1}{2k_0+1}}\|u\|_{L^2} \leq c_0 \|Pu-zu\|_{L^2},
\end{equation}
where $P=P^w(x,hD_x;h)$, $k_0$ is the integer defined in {\rm(\ref{ning3})} and $\Omega_{h}$ denotes the set
\begin{equation}
\label{eq1.7}
\Omega_{h}=\Big\{z \in \comp: \emph{\textrm{Re }}z \leq \frac{1}{C} h^{\frac{2k_0}{2k_0+1}}|z|^{\frac{1}{2k_0+1}}, \ Ch \leq |z| \leq C_0  \Big\}.
\end{equation}
\end{theo}

\bigskip
\noindent
The set $\Omega_h$ defined in (\ref{eq1.7}) is represented on Figure 2. We may also notice that when $z\in \Omega_h$, then Theorem \ref{theo} implies that $z$ is in the resolvent set of $P$, and the resolvent estimate
$$
\left(P-z\right)^{-1}  = {\cal O}\left(h^{-\frac{2k_0}{2k_0+1}}\abs{z}^{-\frac{1}{2k_0+1}}\right): L^2(\real^n)\rightarrow L^2(\real^n)
$$
holds.

\begin{figure}[hhh]\label{des2}
\caption{Set $\Omega_h$.}
\includegraphics[scale=0.7]{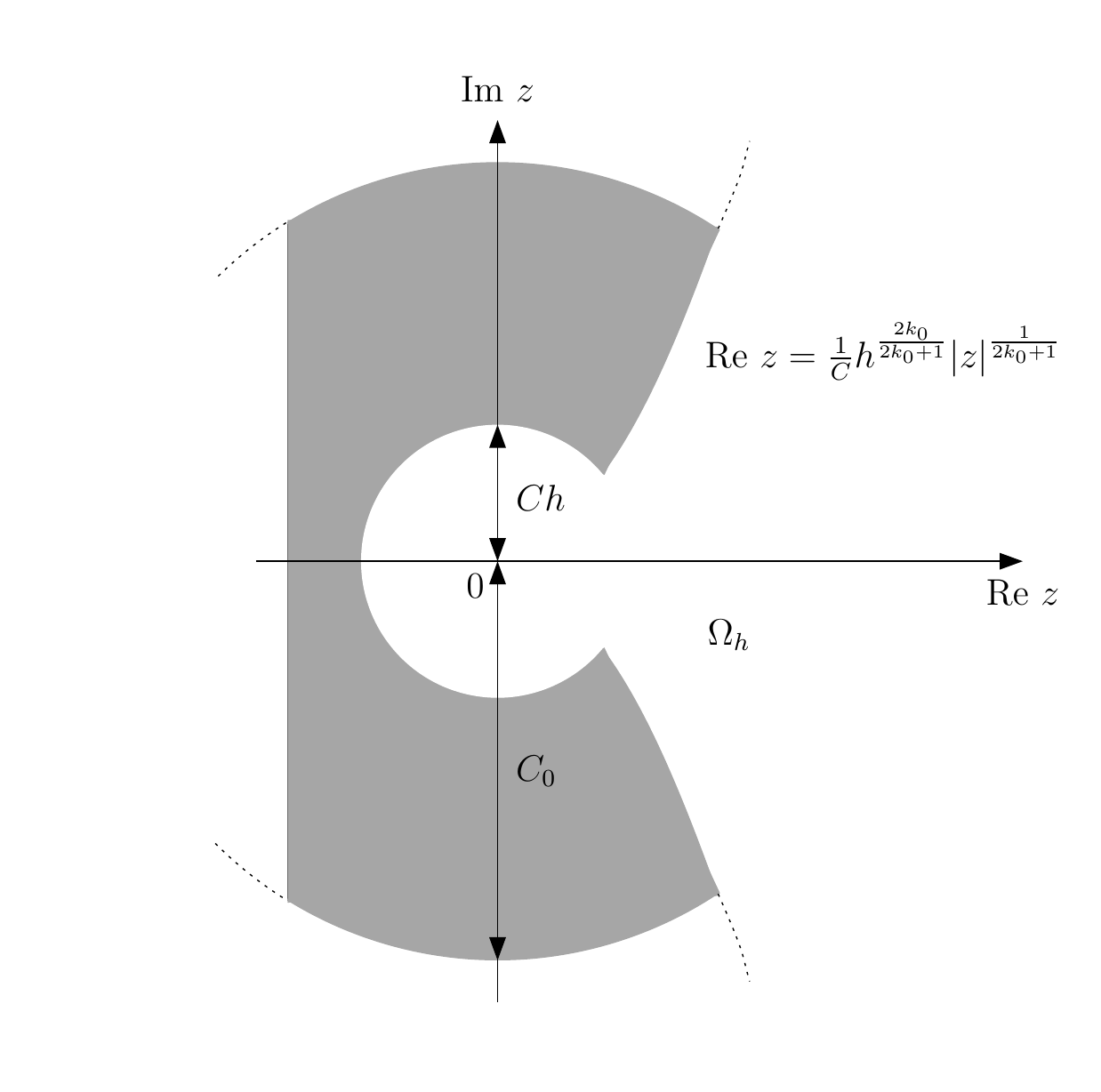}
\end{figure}

\medskip
\noindent
Notice that the quantity $h^{\frac{2k_0}{2k_0+1}}|z|^{\frac{1}{2k_0+1}}$, which appears in the estimate (\ref{ning4}), when $Ch \leq |z| \leq C_0$,
increases when the spectral parameter $z$ moves away from the origin at a rate, which depends on the maximal loss of derivatives $2k_0/(2k_0+1)$
appearing in the subelliptic estimates (\ref{dl1}), fulfilled by the quadratic approximations of the principal symbol at the doubly characteristic points. When the spectral parameter is of the order of magnitude of $h$, we recover the semiclassical hypoelliptic a priori estimate (\ref{eq1.6.7}), proved in~\cite{HiPr2}, with a loss of the full power of the semiclassical parameter. Theorem~\ref{theo} and Theorem~1 in~\cite{HiPr2},
together with the description of the spectrum of $P$, given in Theorem \ref{theo1}, give therefore an almost complete picture of the
spectral properties and the growth of the resolvent norm of a non-selfadjoint semiclassical
pseudodifferential operator with double characteristics fulfilling the assumptions of Theorems~\ref{theo} near the doubly characteristic set. These results underline the basic r\^ole played by the singular space in the analysis of the general structure of double characteristics.

\begin{figure}[hhh]\label{des3}
\caption{The estimate $h^{\frac{2k_0}{2k_0+1}}|z|^{\frac{1}{2k_0+1}}\|u\|_{L^2} \leq c\|Pu-zu\|_{L^2}$ is fulfilled when $z$
belongs to the dark grey region of the figure; whereas the estimate $h\|u\|_{L^2} \leq \|Pu-zu\|_{L^2}$ is fulfilled in the light grey one.}
\includegraphics[scale=0.7]{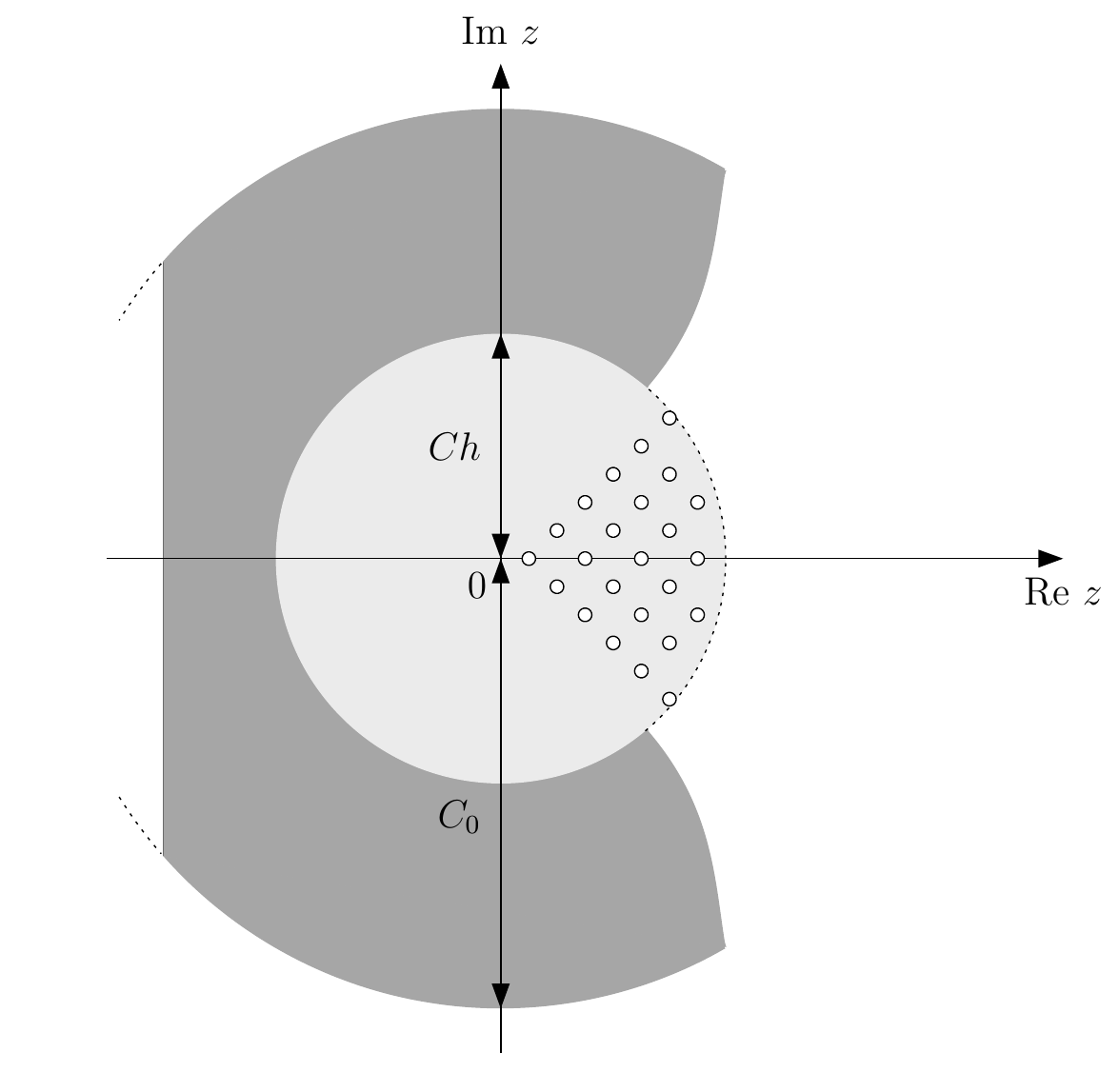}
\end{figure}

Coming back to Theorem~\ref{theo}, we would like to stress the fact that the non-negative integer $k_0$ defined in (\ref{ning3}),
$0 \leq k_0 \leq 2n-1$, measuring the maximal loss of derivatives $2k_0/(2k_0+1)$ appearing in the subelliptic estimates (\ref{dl1})
fulfilled by the quadratic approximations of the principal symbol at doubly characteristic points and the rate of growth of the
resolvent norm when the spectral parameter $z$ moves away from the origin in the estimate (\ref{ning4}); can actually
take any value in the set $\{0,...,2n-1\}$, when $n \geq 1$. Explicit local models for the quadratic approximations of the principal
symbol at doubly characteristic points for which the integer $k_0$ can take any value in the set $\{0,...,2n-1\}$ are given for example
by the following symbols:

\medskip

\begin{itemize}
\item[-] Case $k_0=0$: According to the definition of the Hamilton map, this is the case of any quadratic symbol $q$ with a positive definite real part $\textrm{Re }q > 0$.
\item[-] Case $k_0=1$: Consider a Fokker-Plank operator with a nondegenerate quad\-ra\-tic potential ten\-so\-ri\-zed with 
a harmonic oscillator in other symplectic variables
$$\xi_2^2+x_2^2+i(x_2 \xi_1-x_1 \xi_2)+\sum_{j=3}^n(\xi_j^2+x_j^2).$$
\item[-] Case $k_0=2p$, with $1 \leq p \leq n-1$: Consider
\begin{multline*}
\xi_1^2+x_1^2+i(\xi_1^2+2x_2\xi_1+\xi_2^2+2 x_3\xi_2+....+\xi_p^2+2x_{p+1}\xi_p+\xi_{p+1}^2)\\
+\sum_{j = p+2}^n(\xi_j^2+x_j^2).
\end{multline*}
\item[-] Case $k_0=2p+1$, with $1 \leq p \leq n-1$: Consider
$$x_1^2+i(\xi_1^2+2x_2\xi_1+\xi_2^2+2 x_3\xi_2+....+\xi_p^2+2x_{p+1}\xi_p+\xi_{p+1}^2)+\sum_{j = p+2}^n(\xi_j^2+x_j^2).$$
\end{itemize}
We refer the reader to \cite{Karel11} for more details concerning those examples.

\bigskip
\noindent
\textit{Remark.} The basic r\^ole played by conditions of subelliptic type for the understanding of resolvent estimates for
non-selfadjoint operators of principal type was first stressed in \cite{DeSjZw}. See also \cite{Karel06,mz} for specific cases.
These results were recently improved by W.~Bordeaux Montrieux in a model situation \cite{Bordeaux} and in the general case by
J.~Sj\"ostrand in \cite{Sj10}.

\medskip
\noindent
In~\cite{HeSjSt}, the authors obtain a result analogous to Theorem~\ref{theo1} and a resolvent estimate similar to (\ref{ning4}), in the case when
$k_0=1$. These results are obtained using assumptions of subelliptic type for the principal symbol of the operator, both locally near the doubly
characteristic points, and at infinity. Our analysis does not consider such a general situation where the ellipticity may fail both locally and at
infinity. The purpose of the present work, as well as of~\cite{HiPr2}, is to understand deeper the phenomena occurring near the doubly
characteristic set, and therefore we simplify parts of the analysis of~\cite{HeSjSt} by requiring a property of ellipticity at infinity
(\ref{eq1.5}) for the real part of the principal symbol $p_0$, whereas we weaken the assumptions of subelliptic type at the doubly
characteristic points. The assumption of subelliptic type for the principal symbol $p_0$ of the operator near a doubly characteristic point, say
here $X_0=0$,
$$
\exists \eps_0>0, \ \textrm{Re }p_0(X)+\eps_0H_{\textrm{Im}p_0}^2\textrm{Re }p_0(X) \sim |X|^2,
$$
made in~\cite{HeSjSt}, implies (See Section~4 in~\cite{HiPr2}) that the singular space $S$ associated to the quadratic approximation $q$
of the principal symbol $p_0$ at $X_0=0$ is reduced to $\{0\}$. More specifically, the singular space $S$ is equal to zero after the
intersection of exactly two kernels,
$$
S=\textrm{Ker}(\textrm{Re }F) \cap \textrm{Ker}\big[\textrm{Re }F(\textrm{Im }F)\big] \bigcap \real^{2n}=\{0\}.
$$
This explains why the integer $k_0$ is equal to $1$ in the case studied in~\cite{HeSjSt}.

\bigskip
\noindent
In the proof of Theorem \ref{theo1}, we rely upon the techniques developed in~\cite{HeSjSt},~\cite{HiPr1},~\cite{HiPr2}, and similarly
to~\cite{HeSjSt}, the proof proceeds by solving a globally well-posed Grushin problem for the operator $P$ in a suitable microlocally weighted
$L^2$--space, introduced in~\cite{HiPr2}. The main technical tool in the first part of the paper is therefore a systematic use of the
FBI--Bargmann transformation as well as of the associated weighted spaces of holomorphic functions.

\medskip
\noindent
The proof of Theorem~\ref{theo} uses elements of the Wick calculus, whose main features are recalled in the appendix (Section~\ref{Wick}).
This proof also depends crucially on the construction of weight functions performed in~\cite{Karel11} (Proposition~2.0.1) for the quadratic
approximations of the principal symbol at the doubly characteristic points. The method used in this proof, by starting with weights
built for quadratic symbols in order to deal with the general doubly characteristic case, largely accounts for the assumption (\ref{re1}).
We shall need this assumption in our proof of Theorem~\ref{theo}. Nevertheless, this hypothesis may be relevant only technically.

\medskip
\noindent
The plan of the paper is as follows. In Section 2, we study quadratic differential operators with quadratic symbols $q$, elliptic
along the associated singular spaces, and derive some Gaussian decay estimates for the generalized eigenfunctions, thereby
completing the corresponding discussion in~\cite{HiPr1}. This study is instrumental in Section 3, devoted to the construction of a globally
well-posed Grushin proof for the operator $P$ and to the proof of Theorem \ref{theo1}. Theorem \ref{theo} is established in Section 4.
As alluded to above, the proof makes use of some elements of the Wick calculus, and the relevant facts concerning those techniques
are reviewed in the appendix.

\medskip
\noindent
{\bf Acknowledgments}. The research of the first author is partially supported by the National Science Foundation under grant DMS-0653275 and by
the Alfred P. Sloan Research Fellowship. Part of this projet was conducted when the two authors visited
Universit\'e de Rennes in June of 2009. It is a great pleasure for them to thank Francis Nier for the invitation and for
the inspiring discussions. The authors are also very grateful to San V\~u Ng\d{o}c for the generous hospitality in Rennes. 

\section{Gaussian decay of eigenfunctions in the quad\-ra\-tic case}
\setcounter{equation}{0}
In this section we shall be concerned with a quadratic form $q$ on $\real^{2n}$ such that $\Re q\geq 0$ and with $q$ being elliptic along the
associated singular space $S$, introduced in (\ref{h1}). It follows then from~\cite{HiPr1} (Section~1.4.1) that the singular space
$S\subset \real^{2n}$ is symplectic. We have the following decomposition,
\begeq
\label{eq2.1}
\real^{2n}=S^{\sigma \perp}\oplus S,
\endeq
where $S^{\sigma \perp}$ is the orthogonal space of $S$ with respect to the symplectic form $\sigma$ in $\real^{2n}$, and let us
recall from~\cite{HiPr1} (Section~2) that we have linear symplectic coordinates
$(x',\xi')$ in $S^{\sigma \perp}$ and $(x'',\xi'')$ in $S$, respectively,
such that if
\begeq
\label{eq2.1.0}
X=(x,\xi)=(X';X'')=(x',\xi';x'',\xi'')\in \real^{2n}=\real^{2n'}\times \real^{2n''},
\endeq
then
\begeq
\label{eq2.2}
q(x,\xi)=q_1(x',\xi')+i q_2(x'',\xi''), \quad q_1  = q|_{S^{\sigma \perp}},\quad i q_2  = q|_S.
\endeq
We know furthermore from~\cite{HiPr1} (Proposition~2.0.1) that the symplectic coordinates may be chosen such that the elliptic quadratic form
$q_2$ satisfies
\begeq
\label{eq2.2.1}
q_2(x'',\xi'') = \eps_0 \sum_{j=1}^{n''} \frac{\lambda_j}{2} \left( {x''_j}^2 + {\xi''_j}^2\right),\quad \lambda_j>0, \quad
\eps_0 \in \{\pm 1\},
\endeq
while $q_1$ enjoys the following averaging property: for each $T > 0$, the quadratic form
\begeq
\label{eq2.3}
\langle{\Re q_1}\rangle_T(x',\xi')=\frac{1}{T}\int_{0}^{T} \Re q_1 \left(\exp(tH_{{\rm Im}\, q_1})(x',\xi')\right)\,dt
\endeq
is positive definite in $(x',\xi')$. In what follows, in order to fix the ideas, we take $\eps_0=1$ in (\ref{eq2.2.1}).

\medskip
\noindent
Following~\cite{HiPr2} (Section~2), let us introduce the quadratic weight function,
\begeq
\label{eq2.4}
G_0(X) = - \int J\left(-\frac{t}{T}\right) \Re q \left(\exp(t H_{{\rm Im}\, q})(X)\right)\,dt,\quad T> 0,
\endeq
where $J$ is a compactly supported piecewise affine function satisfying
$$
J'(t) = \delta(t) - 1_{[-1,0]}(t),
$$
and $1_{[-1,0]}$ the characteristic function of the set $[0,1]$.
It follows that
\begin{equation}
\label{eq2.5}
H_{{\rm Im}\, q} G_0 =\langle{\Re q}\rangle_{T,{\rm Im}\, q} -\Re q,
\end{equation}
where
$$
\langle \Re q \rangle_{T,{\rm Im}\, q}(X)=\frac{1}{T}\int_{0}^{T} \Re q (\exp(tH_{{\rm Im}\, q})(X))\,dt.
$$
From (\ref{eq2.2}) and (\ref{eq2.2.1}) we see that $G_0$ is a function of $X'$ only, so that $G_0 = G_0 (X')$, $X'=(x',\xi')\in \real^{2n'}$.
Following~\cite{HeSjSt} and~\cite{HiPr2}, we shall therefore consider an IR-deformation of the real phase space $S^{\sigma \perp}=\real^{2n'}$,
associated to the quadratic weight $G_0$, viewed as a function on $\real^{2n'}$. Let us set
\begeq
\label{eq2.6}
\Lambda_{\delta}  = \{X' + i\delta H_{G_0}(X');\, X'\in \real^{2n'}\}\subset \comp^{2n'},\quad 0\leq \delta \leq 1.
\endeq
We then know that for all $\delta>0$ small enough, $\Lambda_{\delta}$ is a linear IR-manifold, and, as explained for instance in~\cite{HeHiSj}
(Section~4), there exists a linear canonical transformation
\begeq
\label{eq2.7}
\kappa_{\delta}: \real^{2n'} \rightarrow \Lambda_{\delta},
\endeq
such that
\begeq
\label{eq2.8}
\kappa_{\delta}(X')  = X' + i\delta H_{G_0}(X') + {\cal O}(\delta^2 \abs{X'}).
\endeq

\medskip
\noindent
We introduce next the standard FBI-Bargmann transformation along $S^{\sigma \perp} \simeq \real^{2n'}$,
\begeq
\label{eq2.9}
T'u (x') = \widetilde{C} h^{-3n'/4} \int e^{\frac{i}{h}\varphi(x',y')} u(y')\,dy',\quad x'\in \comp^{n'},\quad \widetilde{C}>0,
\endeq
where $\varphi(x',y')=\frac{i}{2}(x'-y')^2$. Associated to $T'$ there is a complex linear canonical transformation
\begeq
\label{eq2.10}
\kappa_{T'}: \comp^{2n'} \ni (y',\eta')\mapsto (x',\xi')=(y'-i\eta',\eta')\in \comp^{2n'},
\endeq
mapping the real phase space $\real^{2n'}$ onto the linear IR-manifold
\begeq
\label{eq2.11}
\Lambda_{\Phi_0}=\Big\{ \Big(x',\frac{2}{i} \frac{\partial \Phi_0}{\partial x'}(x')\Big) : x' \in \comp^{n'}\Big\},
\endeq
where
$$\Phi_0(x') =\frac{1}{2} \left(\Im x'\right)^2.$$
For a suitable choice of $\widetilde{C}>0$ in (\ref{eq2.9}), we know that the map $T'$ takes $L^2(\real^{n'})$ unitarily onto $H_{\Phi_0,h}(\comp^{n'})$.
Here and in what follows, when $\Phi\in C^{\infty}(\comp^{n'})$ is a suitable smooth strictly plurisubharmonic weight function close to $\Phi_0$
in (\ref{eq2.11}), we shall let $H_{\Phi,h}(\comp^{n'})$ stand for the closed subspace of $L^2(\comp^{n'}; e^{-\frac{2\Phi}{h}}
L(dx'))$, consisting of functions that are entire holomorphic. The integration element $L(dx')$ stands here for the Lebesgue
measure on $\comp^{n'}$.

\medskip
\noindent
Following~\cite{HiPr2} (Section~3), we write next
\begin{equation}
\label{eq2.12}
\kappa_{T'}(\Lambda_{\delta})=\Lambda_{\Phi_{\delta}}:=\Big\{\Big(x',\frac{2}{i}\frac{\partial
\Phi_{\delta}}{\partial x'}(x')\Big) ; x'\in \comp^{n'}\Big\},
\end{equation}
for $0\leq \delta \leq \delta_0$  with $\delta_0>0$ small enough,
where $\Phi_{\delta}(x')$ is a strictly plurisubharmonic quadratic form on $\comp^{n'}$, given by
\begeq
\label{eq2.12.1}
\Phi_{\delta}(x') = {\rm v.c.}_{(y',\eta')\in {\bf C}^{n'}\times {\bf R}^{n'}} \left(-\Im \varphi(x',y') - (\Im y')\cdot \eta' + \delta G_0(\Re y', \eta')\right).
\endeq
The unique critical point $(y'(x'),\eta(x'))$ giving the corresponding critical value in (\ref{eq2.12.1}) satisfies
\begeq
\label{eq2.12.2}
y'(x') = \Re x' +{\cal O}(\delta\abs{x'}),\quad \eta'(x') = -\Im x' + {\cal O}(\delta \abs{x'}),
\endeq
and as in~\cite{HiPr1},~\cite{HiPr2}, we obtain that
\begin{equation}
\label{eq2.13}
\Phi_{\delta}(x')=\Phi_0(x')+\delta G_0(\Re x', -\Im x')+\mathcal{O}(\delta^2 \abs{x'}^2).
\end{equation}

\medskip
\noindent
Let us set $Q_1 =q_1^w(x',hD_{x'})$ and recall from~\cite{Sj95} the exact Egorov property
\begin{equation}
\label{2.14}
T' Q_1 u= \widetilde{Q}_1 T' u,\ u\in \mathcal{S}(\real^{n'}),
\end{equation}
where $\widetilde{Q}_1$ is a semiclassical quadratic differential operator on $\comp^{n'}$ whose Weyl symbol $\widetilde{q}_1$ satisfies
\begin{equation}
\label{2.15}
\widetilde{q}_1 \circ \kappa_{T'} = q_1 ,
\end{equation}
with $\kappa_{T'}$ given in (\ref{eq2.10}).

\medskip
\noindent
\medskip
Continuing to follow~\cite{Sj95}, let us also recall that when realizing $\widetilde{Q}_1$ as an unbounded operator on
$H_{\Phi_0,h}(\comp^{n'})$, we may first use the contour integral representation
$$
\widetilde{Q}_1 u(x') = \frac{1}{(2\pi h)^{n'}}\int\!\!\!\int_{\theta'=\frac{2}{i}\frac{\partial \Phi_0}{\partial
x'}\left(\frac{x'+y'}{2}\right)} e^{\frac{i}{h}(x'-y')\cdot \theta'}
\widetilde{q}_1\Big(\frac{x'+y'}{2},\theta'\Big)u(y')\,dy'\,d\theta',$$
and then, using that the symbol $\widetilde{q}_1$ is holomorphic, by a contour
deformation we obtain the following formula for $\widetilde{Q}_1$ as an unbounded operator on $H_{\Phi_0,h}(\comp^{n'})$,
\begin{equation}
\label{eq2.16}
\widetilde{Q}_1 u(x') = \frac{1}{(2\pi h)^{n'}}\int\!\!\!\int_{\theta'=\frac{2}{i}\frac{\partial \Phi_0}{\partial
x'}\left(\frac{x'+y'}{2}\right)+it\overline{(x'-y')}} e^{\frac{i}{h}(x'-y')\cdot \theta'}
\widetilde{q}_1\Big(\frac{x'+y'}{2},\theta'\Big)u(y')\,dy'\,d\theta',
\end{equation}
for any $t>0$. Furthermore, the operator $\widetilde{Q}_1$ can also be viewed as an unbounded operator
\begeq
\label{eq2.17.0}
\widetilde{Q}_1: H_{\Phi_{\delta},h}(\comp^{n'})\rightarrow H_{\Phi_{\delta},h}(\comp^{n'}),
\endeq
defined for $0<\delta\leq \delta_0$, with $\delta_0>0$ sufficiently small. Indeed, when defining the operator in (\ref{eq2.17.0}), it suffices to make
a contour deformation in (\ref{eq2.16}) and set
\begin{equation}
\label{eq2.17.001}
\widetilde{Q}_1 u(x') = \frac{1}{(2\pi h)^{n'}}\int\!\!\!\int_{\theta'=\frac{2}{i}\frac{\partial \Phi_{\delta}}{\partial
x'}\left(\frac{x'+y'}{2}\right)+it\overline{(x'-y')}} e^{\frac{i}{h}(x'-y')\cdot \theta'}
\widetilde{q}_1\Big(\frac{x'+y'}{2},\theta\Big)u(y')\,dy'\,d\theta',
\end{equation}
for any $t>0$.
We then know from the general theory~\cite{MelinSj}, \cite{Sj82}, that the operator in (\ref{eq2.17.0}) is unitarily equivalent to the quadratic
operator on $L^2(\real^{n'})$, whose Weyl symbol is given by the quadratic form
\begeq
\label{eq2.17.01}
X' \mapsto  q_1\left(\kappa_{\delta}(X')\right),\quad X'\in \real^{2n'},
\endeq
with $\kappa_{\delta}$ introduced in (\ref{eq2.7}), (\ref{eq2.8}). In particular, using (\ref{eq2.3}), (\ref{eq2.5}), and (\ref{eq2.8}), we see as
in \cite{HiPr1} (p.827) that the real part of the quadratic form in (\ref{eq2.17.01}) is positive definite, and from~\cite{HiPr1} (p.828) we also
know that the spectrum of $\widetilde{Q}_1$ acting on $H_{\Phi_0,h}(\comp^{n'})$ agrees with the spectrum of $\widetilde{Q}_1$ acting on
$H_{\Phi_{\delta},h}(\comp^{n'})$, for all $0<\delta\leq \delta_0$,
$\delta_0>0$ small enough, including the multiplicities. For future reference, let us recall from~\cite{HiPr1} the explicit description of the
spectrum of $\widetilde{Q}_1$, which is given by
\begeq
\label{eq2.17.02}
{\rm Spec}(\widetilde{Q}_1) =
\left\{ h \mathop{\sum_{\lambda \in \sigma(F_1)}}_{{\rm Im} \lambda > 0}
(r_\lambda + 2 k_\lambda)\frac{\lambda}{i},\, k_\lambda \in \nat \right\}.
\endeq
Here, $F_1$ is the Hamilton map associated to the quadratic form $q_1$ and $r_\lambda$ is the dimension of the generalized eigenspace of $F_1$ in
$\comp^{2n'}$ corresponding to the eigenvalue $\lambda\in \comp$ of the Hamilton map $F_1$.

\medskip
\noindent
In the remainder of this section, we shall be concerned exclusively with the case of $(h=1)$ quantization, and we shall then write
$H_{\Phi_0}(\comp^{n'})=H_{\Phi_0,h=1}(\comp^{n'})$, and similarly for $H_{\Phi_{\delta}}(\comp^{n'})$. The following result is a slight
generalization of the corresponding statement from~\cite{HiPr1}.
\begin{prop}
\label{prop2.1}
There exists $\eta>0$ and $\delta_0>0$ small enough, such that the generalized eigenvectors $u$ of the operators
$$
\widetilde{Q}_1(x',D_{x'}): H_{\Phi_0}(\comp^{n'}) \rightarrow H_{\Phi_0}(\comp^{n'})
$$
and
$$
\widetilde{Q}_1(x',D_{x'}): H_{\Phi_\delta}(\comp^{n'}) \rightarrow H_{\Phi_\delta}(\comp^{n'}),\quad 0<\delta \leq \delta_0,
$$
agree and satisfy
\begeq
\label{eq2.17.1}
u\in H_{\Phi_0-\eta \abs{x'}^2}(\comp^{n'}).
\endeq
\end{prop}
\begin{proof}
The statement of the proposition was established in the work~\cite{HiPr1}, in the case when $u$ is an eigenvector of $\widetilde{Q}_1$.
When treating the case of generalized eigenvectors, we may argue in a way similar to~\cite{HiPr1} (p.829-831), and consider the restriction of the
heat semigroup, viewed as a bounded operator,
\begeq
\label{eq2.18}
e^{-t\widetilde{Q}_1} : H_{\Phi_0}(\comp^{n'}) \rightarrow H_{\Phi_t}(\comp^{n'}),\quad 0 < t\leq t_0,
\endeq
$t_0>0$ small enough, to a generalized eigenspace $E_{\lambda_0}\subset H_{\Phi_0}(\comp^{n'})$ of $\widetilde{Q}_1$, associated to an eigenvalue
$\lambda_0$. The space $E_{\lambda_0}$ is finite-dimensional, and the restriction of $\widetilde{Q}_1-\lambda_0$ to $E_{\lambda_0}$ is nilpotent.
It was shown in \cite{HiPr1} (Lemma~3.1.2) that for each $t>0$ small enough, there exists $\alpha=\alpha(t)>0$ such that the quadratic form
$\Phi_t$ satisfies
$$
\Phi_t(x') \leq \Phi_0(x') - \alpha \abs{x'}^2,\quad x'\in \comp^{n'}.
$$
Notice that the map $e^{-t\widetilde{Q}_1}: E_{\lambda_0} \rightarrow E_{\lambda_0}$ is bijective for any $t\geq 0$.
Indeed, the generalized eigenspace $E_{\lambda_0}$ is stable under the action of the operator $\widetilde{Q}_1$ and its restriction to this
finite-dimensional space
$$
\widetilde{Q}_1|_{E_{\lambda_0}} : E_{\lambda_0} \rightarrow E_{\lambda_0},
$$
is a bounded operator. This implies that the restriction of the semigroup to the space $(e^{-t\widetilde{Q}_1})|_{E_{\lambda_0}}$ coincides with
the exponential of the bounded operator $-t\widetilde{Q}_1|_{E_{\lambda_0}}$, which is always bijective.
It follows therefore that the generalized eigenvectors $u\in H_{\Phi_0}(\comp^{n'})$ of $\widetilde{Q}_1$ acting on $H_{\Phi_0}(\comp^{n'})$, belong to
$H_{\Phi_{\delta}}(\comp^{n'})$, for $\delta>0$ small enough, and satisfy (\ref{eq2.17.1}). Considering the action of the heat semigroup on
the corresponding generalized eigenspace of the operator $\widetilde{Q}_1$ acting on $H_{\Phi_{\delta}}(\comp^{n'})$ and repeating the arguments
following the statement of Lemma 3.1.2 in~\cite{HiPr1}, we obtain the statement of the proposition.
\end{proof}

\medskip
\noindent
Having obtained the exponential decay properties of the generalized eigenvectors of $\widetilde{Q}_1$, we return to the full quadratic operator
$Q = q^w(x,D_x)$ in (\ref{eq2.2}), and introduce the corresponding quadratic differential operator $\widetilde{Q}$ on the FBI transform side, given by
$$
T Q u = \widetilde{Q} Tu,\quad u\in {\cal S}(\real^n).
$$
Here the full FBI-Bargmann transformation $T$ is given by
$$
T = T' \otimes T'': L^2(\real^n) = L^2(\real^{n'})\otimes L^2(\real^{n''})\rightarrow H_{\Phi_0}(\comp^{n'})\otimes H_{\Phi_0}(\comp^{n''}) = H_{\Phi_0}(\comp^n),
$$
with the partial transform $T^{''}$ along the singular space $S$ being defined similarly to (\ref{eq2.9}). Associated to $T''$ and to $T$, we have the linear canonical transformations $\kappa_{T''}$ and $\kappa_T$, with $\kappa_T = \kappa_{T'}\otimes \kappa_{T''}$, so that $\kappa_T(y,\eta) = (y-i\eta,\eta)$. The splitting of the coordinates (\ref{eq2.1.0}) induces, by means of $\kappa_T$, the corresponding splitting of the coordinates in $\comp^n$, so that we can write $x=(x',x'')\in \comp^n = \comp^{n'}\times \comp^{n''}$. We have, in view of (\ref{eq2.2}),
\begeq
\label{eq2.19}
\widetilde{Q}(x,D_x) = \widetilde{Q}_1(x',D_{x'}) + i \widetilde{Q}_2(x'',D_{x''}),
\endeq
where the symbol $\widetilde{q}_2$ of the quadratic operator $\widetilde{Q}_2(x'',D_{x''})$ is given by $\widetilde{q}_2 = q_2 \circ \kappa_{T''}^{-1}$.

\medskip
\noindent
We shall be concerned with the generalized eigenfunctions of the operator $\widetilde{Q}(x,D_x)$ in (\ref{eq2.19}) acting on the weighted space
$$
H_{\Phi_{\delta}}(\comp^{n}) = H_{\Phi_{\delta}}(\comp^{n'})\otimes H_{\Phi_0}(\comp^{n''}),
$$
with $\delta>0$ small enough fixed. Here in the left hand side,
$$
\Phi_{\delta}(x)=\Phi_{\delta}(x')+\Phi_0(x''),
$$
and an application of (\ref{eq2.11}) and (\ref{eq2.13}) shows that
$$
\Phi_{\delta}(x) = \Phi_0(x) + \delta G_0(\Re x', -\Im x') + {\cal O}(\delta^2 \abs{x'}^2).
$$

\medskip
\noindent
Let us recall from~\cite{HiPr1} (p.843) that the spectrum of $\widetilde{Q}(x,D_x)$ is given by
$$
\sigma(\widetilde{Q}(x,D_x)) = \sigma(\widetilde{Q}_1(x',D_{x'})) + i\sigma(\widetilde{Q}_2(x'',D_{x''})),
$$
with the spectrum of $\widetilde{Q}_1(x',D_{x'})$ given in (\ref{eq2.17.02}),
and furthermore, from (\ref{eq2.2.1}), we know that the spectrum of $\widetilde{Q}_2(x'',D_{x''})$ consists of the eigenvalues of the form
$$
\mu_{\alpha''}  = \sum_{j=1}^{n''} \frac{\lambda_j}{2}\left(2\alpha_j''+1\right),\quad \alpha''\in \nat^{n''}.
$$
The corresponding eigenfunctions are given by
\begeq
\label{eq2.20}
\Phi_{\alpha''}(x'') = (T''\varphi_{\alpha''})(x''),
\endeq
where
$$
\varphi_{\alpha''}(y'') = H_{\alpha''}(y'') e^{-(y'')^2/2}
$$
are the Hermite functions, with $H_{\alpha''}(y'')$ being the Hermite polynomials on $\real^{n''}$. It is clear that the eigenfunctions
$\Phi_{\alpha''}(x'')$ form an orthonormal basis of $H_{\Phi_0}(\comp^{n''})$, and a straightforward computation shows that the functions
$\Phi_{\alpha''}(x'')$ are of the form
$$
\Phi_{\alpha''}(x'') = p_{\alpha''}(x'') e^{-(x'')^2/4},
$$
where $p_{\alpha''}(x'')$ is a holomorphic polynomial on $\comp^{n''}$. In particular, we have
\begeq
\label{eq2.21}
\Phi_{\alpha''} \in H_{\Phi_0-\eta \abs{x''}^2}(\comp^{n''}),
\endeq
for some fixed $\eta>0$.

\medskip
\noindent
Let $u\in H_{\Phi_{\delta}}(\comp^n)$, and let us write
$$
u(x',x'') = \sum_{\alpha''\in {\bf N}^{n''}} u_{\alpha''}(x') \Phi_{\alpha''}(x'').
$$
Using that
\begeq
\left(\widetilde{Q}(x,D_x)-\lambda\right) u = \sum_{\alpha''\in {\bf N}^{n''}} \left[(\widetilde{Q}_1(x',D_{x'})+i\mu_{\alpha''}-\lambda)u_{\alpha''}(x')\right]\Phi_{\alpha''}(x''),
\endeq
we see that $u$ is a generalized eigenvector of $\widetilde{Q}(x,D_x)$ corresponding to an eigenvalue $\lambda\in \comp$, precisely when
$u$ is of the form
\begeq
\label{eq2.22}
u(x',x'') = \sum_{\alpha''} u_{\alpha''}(x')\Phi_{\alpha''}(x''),
\endeq
where the summation extends over all $\alpha''\in \nat^{n''}$ for which
$$
\lambda-i\mu_{\alpha''}\in \sigma(\widetilde{Q}_1(x',D_{x'})),
$$
and $u_{\alpha''}(x')\in H_{\Phi_{\delta}}(\comp^{n'})$ is a generalized eigenvector of $\widetilde{Q}_1(x',D_{x'})$ associated to the eigenvalue $\lambda-i\mu_{\alpha''}$. Since, according to (\ref{eq2.17.02}), $\sigma(\widetilde{Q}_1)$ is contained in a proper closed cone in $\comp$ of the form
$\abs{\Im z}\leq C\Re z$, $C>0$, it follows that the sum in (\ref{eq2.22}) contains a fixed finite number of terms, when $\abs{\lambda} = {\cal O}(1)$.
Combining Proposition 2.1, (\ref{eq2.21}), and (\ref{eq2.22}), we obtain the following result, which summarizes the discussion pursued in this section.

\begin{prop}
\label{prop2.2}
There exists $\eta>0$ such that for all $0\leq \delta\leq \delta_0$, with $\delta_0>0$ small enough, the generalized eigenvectors $u$ of the quadratic
operator $\widetilde{Q}(x,D_x)$ acting on $H_{\Phi_{\delta}}(\comp^n)$, satisfy
$$
u \in H_{\Phi_0-\eta \abs{x}^2} (\comp^n).
$$
\end{prop}

\medskip
\noindent
{\it Remark}. The discussion in this section, together with the corresponding analysis in Section 3 in~\cite{HiPr1}, can be considered as a natural
generalization of Remark 11.7 in~\cite{HeSjSt}. For future reference, let us also remark that from~\cite{HeSjSt},~\cite{HeHiSj},~\cite{sjostrand},
we know that the generalized eigenfunctions $u$ of the operator
$\widetilde{Q}(x,D_x)$ are such that the inverse FBI transform $T^{-1}u \in L^2(\real^n)$ is of the form
\begeq
\label{eq2.23}
T^{-1} u = p(x) e^{i\Phi(x)},
\endeq
where $p$ is a polynomial on $\real^n$ and $\Phi(x)$ is a complex quadratic form, and according to Proposition \ref{prop2.2}), we have $\Im \Phi>0$.
Furthermore, the positive Lagrangian subspace $\{(x,\Phi'(x));\, x\in \comp^n\}$ is the stable outgoing manifold for the Hamilton flow of the
quadratic form
$$
(x,\xi) \mapsto \frac{1}{i} e^{-i\theta}q(x,\xi),\quad (x,\xi)\in \real^{2n},
$$
where $\theta>0$ is sufficiently small but fixed.

\section{Global Grushin problem}
\setcounter{equation}{0}
Throughout this section, we shall make the simplifying assumption that the integer $N$ introduced in (\ref{eq1.6}) satisfies $N=1$, and that the
corresponding doubly characteristic point is $X_1=(0,0)\in \real^{2n}$. This assumption serves merely to simplify the notation in the proofs and does
not cause any loss of generality. In particular, we write
$$
p_0(X) = q(X) + {\cal O}(X^3),
$$
where $q$ is a quadratic form, to which Proposition \ref{prop2.2} applies.

\medskip
\noindent
When proving Theorem \ref{theo1}, it will be convenient to work with symbols in the class $S(1)$, bounded together with all of their derivatives,
similarly to what was done in~\cite{HiPr2}. Let us begin this section by describing therefore a reduction to the case when $m=1$. When doing so,
we notice that for all $h>0$ sufficiently small, the operator
$$
P+1 : H(m)\rightarrow L^2(\real^n)
$$
is bijective, and by an application of Beals's lemma, we know that $(P+1)^{-1} \in {\rm Op}_h(S(\left(\frac{1}{m}\right)))$, see~\cite{DiSj}, p.99-100.
Let
$$
\widetilde{P} = (P+1)^{-1} P \in {\rm Op}_h\left(S(1)\right),
$$
with the leading symbol given by
\begeq
\label{eq1.10}
\widetilde{p}_0 = \frac{p_0}{p_0+1}.
\endeq
Furthermore, by holomorphic functional calculus~\cite{HeSjIII}, or by an explicit calculation using the Weyl calculus~\cite{DiSj}
(use formula (8.11) p.100), we see that the subprincipal symbol of $\widetilde{P}$ is given by
\begeq
\label{eq1.11}
\widetilde{p}_1  = \frac{p_1}{(p_0+1)^2}.
\endeq
It follows from (\ref{eq1.10}) that the leading symbol $\widetilde{p}_0$ of the bounded $h$-pseudodifferential operator $\widetilde{P}$
satisfies $\Re \widetilde{p}_0\geq 0$, and that $\Re \widetilde{p}_0$ is elliptic near infinity in the class $S(1)$. Furthermore, $\widetilde{p}_0$
vanishes precisely at the origin, with
$$
\widetilde{p}_0(X) = q(X) + {\cal O}(X^3),\quad \widetilde{p}_1(0) = p_1(0).
$$
In order to deduce the asymptotic description of the eigenvalues for the operator $P$ from the corresponding description for the operator $\widetilde{P}$, we notice that the resolvents of $P$ and $\widetilde{P}$ are related as follows, for $z\in {\rm neigh}(0,\comp)$,
$$
\left(\widetilde{P}-z\right)^{-1} = (1-z)^{-1} \left(P-\frac{z}{1-z}\right)^{-1}(P+1).
$$
Hence, $z\in {\rm neigh}(0,\comp)$ is an eigenvalue of $\widetilde{P}$ precisely when $z/(1-z)$ is an eigenvalue of $P$, and the multiplicities agree.
In what follows, we shall therefore be concerned exclusively with the case when $m=1$.

\subsection{Grushin problem in the quadratic case}

In this subsection, we shall describe a well-posed Gru\-shin problem for the ellip\-tic quad\-ra\-tic ope\-ra\-tor $\widetilde{Q}(x,D_x)$ de\-fined in
(\ref{eq2.19}), acting on the weigh\-ted space $H_{\Phi_\delta}(\comp^n)$, for $\delta>0$ small enough but fixed. Let $\lambda_0\in \comp$ be an
eigenvalue of $\widetilde{Q}(x,D_x)$, and let $E_{\lambda_0}\subset H_{\Phi_{\delta}}(\comp^n)$ be the corresponding finite-dimensional generalized
eigenspace. According to Proposition \ref{prop2.2}, we have
$$
E_{\lambda_0} \subset H_{\Phi_0-\eta \abs{x}^2}(\comp^n),\quad \eta>0.
$$
Let $e_1,\ldots,\, e_{N_0}$ be a basis for $E_{\lambda_0}$. We shall now introduce a suitable dual basis. When doing so, let
$\widetilde{Q}^* = \widetilde{Q}^*(x,D_x)$ be the adjoint of the operator $\widetilde{Q} = \widetilde{Q}(x,D_x)$ acting on
the space $H_{\Phi_0}(\comp^n)$. Here the closed densely defined quadratic operator $\widetilde{Q}$ is equipped with the domain
$\{u\in H_{\Phi_0}(\comp^n);\, \widetilde{Q}u \in H_{\Phi_0}(\comp^n)\}$.
According to the discussion in~\cite{Mehler}, p.426, we have $\widetilde{Q}^*  = T \overline{q}^w T^{-1}$. Here the Weyl symbol of $\overline{q}^w$
is the quadratic form $X\mapsto \overline{q(X)}$, which has a non-negative real part, and whose restriction to the corresponding singular space, which
is equal to $S$, is elliptic. Let $f_1,\ldots,\, f_{N_0}$, $f_j \in H_{\Phi_0}(\comp^n)$, be the basis for the generalized
eigenspace of the adjoint operator $\widetilde{Q}^*: H_{\Phi_0}(\comp^n) \rightarrow H_{\Phi_0}(\comp^n)$, associated to the eigenvalue
$\overline{\lambda_0}$, which is dual to $e_1,\ldots\,\, e_{N_0}$. An application of Proposition 2.2 shows that the functions $f_j$, $1\leq j \leq N_0$,
satisfy
\begeq
\label{eq3.0}
f_j \in H_{\Phi_0-\eta \abs{x}^2}(\comp^n),\quad \eta>0.
\endeq
In particular, $f_j \in H_{\Phi_{\delta}}(\comp^n)$, for $\delta>0$ small enough, and we have
\begeq
\label{eq3.0.1}
{\rm det}\, ((e_j,f_k)) \neq 0,\quad 0\leq \delta \leq \delta_0,
\endeq
for some $\delta_0>0$ sufficiently small. Here the scalar product in (\ref{eq3.0.1}) is taken in the space $H_{\Phi_{\delta}}(\comp^n)$.

\medskip
\noindent
Let us introduce the operators
$$
R_-: \comp^{N_0}\rightarrow H_{\Phi_{\delta}}(\comp^n)
$$
and
$$
R_+: H_{\Phi_{\delta}}(\comp^n)\rightarrow \comp^{N_0},
$$
given by $R_- u_- = \sum_{j=1}^{N_0} u_-(j)e_j$ and $(R_+u)(j) = (u,f_j)$, with the scalar product taken in the space $H_{\Phi_{\delta}}(\comp^n)$.
Arguing as in Section 11 of~\cite{HeSjSt}, we obtain that for $z \in {\rm neigh}(\lambda_0,\comp)$, the Grushin operator
\begeq
\label{eq3.1}
\left( \begin{array}{ccc}
\widetilde{Q}-z & R_- \\\
R_+ & 0
\end{array} \right): {\cal D}(\widetilde{Q}) \times \comp^{N_0} \rightarrow
H_{\Phi_{\delta}}(\comp^n) \times \comp^{N_0}
\endeq
is bijective. Here ${\cal D}(\widetilde{Q}) = \{u\in H_{\Phi_{\delta}}(\comp^n); \, (1+\abs{x}^2) u \in L^2_{\Phi_{\delta}}(\comp^n)\}$.

\medskip
\noindent
Continuing to follow~\cite{HeSjSt}, we shall now restore the semiclassical parameter $h>0$ and consider the operators
\begeq
\label{eq3.2}
R_{-,h} = {\cal O}(1): \comp^{N_0} \rightarrow H_{\Phi_{\delta},h}(\comp^n),\quad R_{+,h} = {\cal O}(1):  H_{\Phi_{\delta},h} \rightarrow \comp^{N_0},
\endeq
given by
\begeq
\label{eq3.2.1}
R_{-,h} u_- = \sum_{j=1}^{N_0} u_-(j) e_{j,h},\quad \left(R_{+,h}u\right)(j) = (u,f_{j,h}).
\endeq
Here the scalar product in the definition of $R_{+,h}$ is taken in the space $H_{\Phi_{\delta},h}(\comp^n)$, and
$$
e_{j,h}(x) = h^{-n/2} e_j\left(\frac{x}{\sqrt{h}}\right), \quad f_{j,h}(x) = h^{-n/2} f_j\left(\frac{x}{\sqrt{h}}\right).
$$

\medskip
\noindent
With $\widetilde{Q} = \widetilde{Q}(x,hD_x)$, we shall now consider the semiclassical Grushin problem, given by
\begeq
\label{eq3.3}
\left(\widetilde{Q}-hz\right)u + R_{-,h}u_- = v,\quad R_{+,h}v = v_+.
\endeq
Here $z$ varies in a sufficiently small but fixed \neigh{} of the eigenvalue $\lambda_0$.
At this point, we are exactly in the same situation as described in Section 11 of~\cite{HeSjSt} (Proposition~11.5), and arguing exactly as in that paper,
we see that for each $(v,v_+)\in H_{\Phi_{\delta},h}(\comp^n)\times \comp^{N_0}$, the problem
(\ref{eq3.3}) has a unique solution $(u,u_-)\in H_{\Phi_{\delta},h}(\comp^n)\times \comp^{N_0}$ such that
$(1+\abs{x}^2)u \in L^2_{\Phi_{\delta},h}(\comp^n)$. Furthermore, for every $k\in \real$ fixed, the following a priori estimate holds,
\begeq
\label{eq3.4}
\norm{(h+\abs{x}^2)^{1-k}u} + h^{-k} \abs{u_-} \\
\leq {\cal O}(1) \left(\norm{(h+\abs{x}^2)^{-k}v} + h^{1-k}\abs{v_+}\right).
\endeq
Here the norms are taken in the space $L^2_{\Phi_{\delta},h}(\comp^n)$.

\medskip
\noindent
The estimate (\ref{eq3.4}) can subsequently be localized, and we see that the result of Proposition 11.6 of~\cite{HeSjSt} can be applied to our
situation as it stands, since the proof of Proposition 11.6 in~\cite{HeSjSt} only relies on the ellipticity of the quadratic operator $\widetilde{Q}$
acting on $H_{\Phi_{\delta},h}(\comp^n)$, for $\delta>0$ small enough but fixed, together with the decay estimates given in Proposition \ref{prop2.2} and
in (\ref{eq3.0}). We therefore obtain the following result, which summarizes the discussion in this section.

\begin{prop}
\label{prop3.1}
Let $\chi_0\in C^{\infty}_0(\comp^n)$ be fixed, such that $\chi_0=1$ near $x=0$, and let $k\in \real$ be fixed. Then for
$z\in {\rm neigh}(\lambda_0, \comp)$, we have the following estimate for the problem {\rm (\ref{eq3.3})}, valid for all $h>0$ sufficiently small,
\begin{multline}
\label{eq3.5}
\norm{(h+\abs{x}^2)^{1-k}\chi_0 u} + h^{-k}\abs{u_-} \\
\leq {\cal O}(1)\left(\norm{(h+\abs{x}^2)^{-k}\chi_0 v} + h^{1-k}\abs{v_+} +
h^{1/2} \norm{1_K u}\right).
\end{multline}
Here $K$ is a fixed neighborhood of ${\rm supp}(\nabla \chi_0)$ and $1_K$ stands for the characteristic function of this set. The norms in the estimate
{\rm (\ref{eq3.5})} are taken in the space $L^2_{\Phi_{\delta},h}(\comp^n)$.
\end{prop}

\medskip
\noindent
{\it Remark}. When deriving the estimate (\ref{eq3.5}), following~\cite{HeSjSt}, we replace the functions $f_{j,h}$ in the definition of $R_{+,h}$ by
$\chi(x/R\sqrt{h}) f_{j,h}(x)$, where $\chi\in C^{\infty}_0(\comp^n)$, and $R>0$ is sufficiently large fixed.

\subsection{Localization and exterior estimates}
The purpose of this subsection is to study a globally well-posed Grushin problem for the operator $P$ introduced in (\ref{eq1.3}). When doing so, we
shall be concerned with the action of $P$, after an FBI-Bargmann transformation, on a suitable weighted space of holomorphic functions on $\comp^n$.
We shall therefore first proceed to recall the definition and properties of this space, constructed and introduced in~\cite{HiPr2}.

\medskip
\noindent
In Proposition~2 of~\cite{HiPr2}, it was shown that for all $0<\eps \leq \eps_0$, $0<\delta \leq \delta_0$, with $\eps_0>0$, $\delta_0>0$
sufficiently small, there exists a function $G_{\eps}\in C^{\infty}_0(\real^{2n},\real)$, supported in a sufficiently small but fixed \neigh{} of
the origin, such that $G_{\eps} = {\cal O}(\eps)$, $\nabla^2 G_{\eps} = {\cal O}(1)$, and such that for some $C>1$, $\widetilde{C}>1$, we have
$$
\abs{p_0\left(X+i\delta H_{G_{\eps}}(X)\right)}\geq \frac{\delta}{\widetilde{C}}{\rm min}\left(\abs{X}^2,\eps\right),
$$
in the region where $\abs{X}\leq 1/C$. Furthermore, in the region where $\abs{X}\geq \eps^{1/2}$, we have
\begeq
\label{eq4.0}
\Re \left(\left(1-\frac{i c \delta \eps}{\abs{X}^2}\right)p_0 \left(X+i\delta H_{G_{\eps}}(X)\right)\right) \geq
\frac{\delta \eps}{\widetilde{C}},\quad c>0.
\endeq
Here we have also written $p_0$ for an almost analytic extension of the leading symbol $p_0$ of $P$ to a tubular \neigh{} of $\real^{2n}$,
bounded together with all of its derivatives.

\medskip
\noindent
{\it Remark}. For future reference, we may remark that it follows from the construction of the weight function $G_{\eps}$ in \cite{HiPr2}, that in
the region where $\abs{X}^2 \leq \eps/2$,
we have
\begeq
\label{eq4.0.1}
G_{\eps}(X) = G_0(X') +{\cal O}(X^3),
\endeq
where the quadratic form $G_0$ is defined in (\ref{eq2.4}), see remark p.1002 in \cite{HiPr2}.

\medskip
\noindent
Associated with the weight function $G_{\eps}$ there is an IR-manifold
\begeq
\label{eq4.1}
\Lambda_{\delta,\eps} = \left\{X+i\delta H_{G_{\eps}}(X); X\in \real^{2n} \right \},
\endeq
and arguing as in~\cite{HiPr2} (Section 3), we obtain that
\begeq
\label{eq4.2}
\kappa_T\left(\Lambda_{\delta,\eps}\right)=\Lambda_{\Phi_{\delta,\eps}}:=\left \{(x,\xi)\in \comp^{2n};
\xi = \frac{2}{i} \frac{\partial \Phi_{\delta,\eps}}{\partial x}(x)\right \}.
\endeq
Here $\Phi_{\delta,\eps}\in C^{\infty}(\comp^n)$ is a strictly plurisubharmonic function given by
\begeq
\label{eq4.2.1}
\Phi_{\delta,\eps}(x) = {\rm v.c.}_{(y,\eta)\in {\bf C}^{n}\times {\bf R}^{n}} \left(-\Im \varphi(x,y) - (\Im y)\cdot \eta + \delta
G_{\eps}(\Re y, \eta)\right).
\endeq
Uniformly on $\comp^n$, we have
\begeq
\label{eq4.3}
\Phi_{\delta,\eps}(x)=\Phi_0(x)+\delta G_{\eps}(\Re x, -\Im x)+{\cal O}(\delta^2 \eps),
\endeq
and in particular,
\begeq
\label{eq4.3.1}
\Phi_{\delta,\eps} - \Phi_0 = {\cal O}(\delta \eps).
\endeq
We furthermore know that $\Phi_{\delta,\eps}$ agrees with $\Phi_0$ outside a bounded set and that
\begeq
\label{eq4.4}
\nabla\left(\Phi_{\delta,\eps}-\Phi_0\right)={\cal O}(\delta \eps^{1/2}),
\endeq
with $\nabla^2 \Phi_{\delta,\eps}\in L^{\infty}(\comp^n)$, uniformly in $\delta$ and $\eps$.

\medskip
\noindent
In what follows, similarly to~\cite{HiPr2} (Section 3), we shall be concerned with the case when
\begeq
\label{eq4.4.1}
\eps=Ah,
\endeq
when $A\geq 1$ is sufficiently large but fixed, to be
chosen in what follows. As explained in~\cite{HiPr2} (Section 3), following~\cite{HeSjSt}, the $h$--pseudodifferential operator on the FBI--Bargmann transform
side, $\widetilde{P}:= TPT^{-1}$, can therefore be defined as a uniformly bounded operator
$$
\widetilde{P} ={\cal O}(1): H_{\Phi_{\delta,\eps},h}(\comp^n) \rightarrow H_{\Phi_{\delta,\eps},h}(\comp^n),
$$
given, when $u\in H_{\Phi_{\delta,\eps},h}(\comp^n)$, by
\begeq
\label{eq4.5}
\widetilde{P} u(x) = \frac{1}{(2\pi h)^n} \int\!\!\!\int_{\Gamma_{\delta,\eps}(x)} e^{\frac{i}{h}(x-y)\cdot \theta}
\psi(x-y) \widetilde{P}\left(\frac{x+y}{2},\theta\right) u(y)\,dy\,d\theta + Ru.
\endeq
Here $\psi\in C^{\infty}_0(\comp^n)$ is such that $\psi=1$ near $0$ and $\Gamma_{\delta,\eps}(x)$ is the contour given by
$$
\theta = \frac{2}{i} \frac{\partial \Phi_{\delta,\eps}}{\partial x}\left(\frac{x+y}{2}\right)+it_0\overline{(x-y)},\quad t_0>0.
$$
The remainder $R$ in (\ref{eq4.5}) satisfies
$$
R={\cal O}_A(h^{\infty}): L^2(\comp^n; e^{-\frac{2\Phi_{\delta,\eps}}{h}}L(dx))\rightarrow
L^2(\comp^n; e^{-\frac{2\Phi_{\delta,\eps}}{h}}L(dx)).
$$
Also, in (\ref{eq4.5}) we continue to write $\widetilde{P}$ for an almost holomorphic extension of the full symbol
$\widetilde{P}\in S(\Lambda_{\Phi_0},1)$ of $\widetilde{P}$, $\widetilde{P} = P\circ \kappa_T^{-1}$, to a tubular \neigh{} of
$\Lambda_{\Phi_0}$, bounded together with all of its derivatives.

\medskip
\noindent
We shall be concerned with a global Grushin problem for the operator $\widetilde{P}$ in the weighted space $H_{\Phi_{\delta,\eps},h}(\comp^n)$.
In order to exploit the quadratic Grushin problem for $\widetilde{Q}$, described in subsection 3.1, we shall make use of the observation that there exists
a constant $C>0$ such that in the region of $\comp^n$, where
\begeq
\label{eq4.5.0}
\abs{x} \leq \frac{\sqrt{\eps}}{C},
\endeq
the weight function $\Phi_{\delta,\eps}$ is independent of $\eps$, and furthermore, in this region, we have
\begeq
\label{eq4.5.1}
\Phi_{\delta,\eps} = \Phi_{\delta}(x) + {\cal O}(\delta \abs{x}^3).
\endeq
The equality (\ref{eq4.5.1}) is obtained by a straightforward computation, using (\ref{eq2.12.1}), its analogue for the weight
$\Phi_{\delta,\eps}$, given by (\ref{eq4.2.1}), as well as (\ref{eq4.0.1}).

\medskip
\noindent
By making a rescaling in $\eps$, we may and will assume in the following that we have $C=1$ in
(\ref{eq4.5.0}). It follows that in
the region where $\abs{x}\leq \sqrt{\eps}$, the $L^2$--norm associated to the quadratic weight function $\Phi_{\delta}$ can be replaced by the
$L^2$--norm associated to the full weight $\Phi_{\delta,\eps}$, at the expense of a loss which is
$$
\exp({\cal O}(1) A^{3/2}h^{1/2}) = {\cal O}(1),
$$
provided that $A\geq 1$ is taken large but fixed, and $h\in (0,h_0]$, with $h_0>0$ small enough depending on $A$.
We shall therefore replace the fixed cut-off function $\chi_0$ in Proposition \ref{prop3.1} by $\chi_0(x/\sqrt{\eps})$, and following Section 11
of~\cite{HeSjSt}, this can be achieved by a rescaling argument using the change of variables $x=\sqrt{\eps}\widetilde{x}$.
This argument is carried out in detail, see (11.33), in Section 11.3 of~\cite{HeSjSt}, and for future reference, we shall record it here.

\begin{lemma}
\label{lemma4.1}
Let $\chi_0\in C^{\infty}_0(\comp^n)$ be fixed, such that $\chi_0=1$ near $x=0$, and let $k\in \real$ be fixed. Then for
$z\in {\rm neigh}(\lambda_0, \comp)$, we have the following estimate for the Grushin problem {\rm (\ref{eq3.3})}, valid for $h>0$ sufficiently small,
with $\eps = Ah$,
\begin{multline}
\label{eq4.6}
\norm{(h+\abs{x}^2)^{1-k}\chi_0\left(\frac{x}{\sqrt{\eps}}\right) u} + h^{-k}\abs{u_-} \leq
{\cal O}(1) \norm{(h+\abs{x}^2)^{-k}\chi_0\left(\frac{x}{\sqrt{\eps}}\right) v} \\
+ {\cal O}(1)\left(h^{1-k}\abs{v_+} +
\sqrt{\frac{h}{\eps}} \norm{(h+\abs{x}^2)^{1-k}1_K\left(\frac{x}{\sqrt{\eps}}\right) u}\right).
\end{multline}
Here $K$ is a fixed neighborhood of ${\rm supp}(\nabla \chi_0)$ and $1_K$ stands for the characteristic function of this set. The norms in the estimate {\rm (\ref{eq4.6})} are taken in the space $L^2_{\Phi_{\delta},h}(\comp^n)$. According to {\rm (\ref{eq4.5.1})}, all the norms in the estimate
{\rm (\ref{eq4.6})} can be replaced by the norms in the space $L^2_{\Phi_{\delta,\eps},h}(\comp^n)$, for each fixed $A\gg 1$, provided that $h\in (0,h_0]$, with $h_0>0$ small enough, depending on $A$.
\end{lemma}

\bigskip
\noindent
We now come to study the global Grushin problem for the operator $\widetilde{P}-hz$, for $z\in {\rm neigh}(\lambda_0+p_1(0),\comp)$, in the weighted space
$H_{\Phi_{\delta,\eps},h}(\comp^n)$. Here $p_1(0)$ is the value of the subprincipal symbol $p_1(x,\xi)$ of $P$ at the unique doubly characteristic point, $(0,0)\in \real^{2n}$. With the operators $R_{-,h}$ and $R_{+,h}$ introduced in (\ref{eq3.2.1}), let us consider
\begeq
\label{eq4.7}
(\widetilde{P}-hz)u + R_{-,h}u_-=v,\quad R_{+,h}u = v_+,
\endeq
when $(v,v_+)\in H_{\Phi_{\delta,\eps},h}(\comp^n)\times \comp^{N_0}$. Writing the first equation in (\ref{eq4.7}) in the form
$$
(\widetilde{Q}-h(z-p_1(0)))u + R_{-,h}u_- = v + (\widetilde{Q}+h p_1(0)-\widetilde{P})u,
$$
and applying Lemma \ref{lemma4.1} with $k=1/2$, we get, with some constant $C>0$,
\begin{multline}
\label{eq4.8}
\norm{(h+\abs{x}^2)^{1/2}\chi_0\left(\frac{x}{\sqrt{\eps}}\right) u} + h^{-1/2}\abs{u_-} \\
\leq C\norm{(h+\abs{x}^2)^{-1/2}\chi_0\left(\frac{x}{\sqrt{\eps}}\right) v} +
C\norm{(h+\abs{x}^2)^{-1/2}\chi_0\left(\frac{x}{\sqrt{\eps}}\right) (\widetilde{P}-\widetilde{Q}-hp_1(0))u} \\
+ {\cal O}(h^{1/2})\abs{v_+} + C\sqrt{\frac{h}{\eps}} \norm{(h+\abs{x}^2)^{1/2}1_K\left(\frac{x}{\sqrt{\eps}}\right)u}.
\end{multline}
Here the norms are taken in the space $L^2_{\Phi_{\delta,\eps},h}(\comp^n)$, as explained in Lemma \ref{lemma4.1}.
Now, as was already observed and exploited in~\cite{HiPr2}, see (5.7) in Section 5, we have
$$
\norm{(h+\abs{x}^2)^{-1/2}\chi_0\left(\frac{x}{\sqrt{\eps}}\right) (\widetilde{P}-\widetilde{Q}-hp_1(0))u} = {\cal O}_A(h)\norm{u},
$$
and therefore, using also that $h+\abs{x}^2\leq {\cal O}(\eps)$ in the support of the function
$$x\mapsto 1_K(x/\sqrt{\eps}),$$
we get
\begin{multline}
\label{eq4.9}
h^{1/2} \norm{\chi_0\left(\frac{x}{\sqrt{\eps}}\right)u} + h^{-1/2}\abs{u_-} \\
\leq {\cal O}(h^{-1/2})\norm{v} + {\cal O}_A(h)\norm{u} + {\cal O}(h^{1/2})\abs{v_+} +{\cal O}(h^{1/2})\norm{1_K\left(\frac{x}{\sqrt{\eps}}\right)u}.
\end{multline}
It follows from (\ref{eq4.9}) upon squaring that
\begin{multline}
\label{eq4.10}
h \norm{\chi_0\left(\frac{x}{\sqrt{\eps}}\right)u}^2 + h^{-1}\abs{u_-}^2 \\
\leq \frac{{\cal O}(1)}{h}\norm{v}^2 + {\cal O}_A(h^2)\norm{u}^2 + {\cal O}(h)\abs{v_+}^2 +{\cal O}(h)\norm{1_K\left(\frac{x}{\sqrt{\eps}}\right)u}^2.
\end{multline}
The estimate (\ref{eq4.10}) will be instrumental in obtaining the global well-posedness of the Grushin problem (\ref{eq4.7}).

\bigskip
\noindent
When deriving an a priori estimate for the problem (\ref{eq4.7}) away from an ${\cal O}(\sqrt{\eps})$--neighborhood of the doubly characteristic point $x=0 \in \comp^n$, we shall
proceed very much in the spirit of Section 6 in~\cite{HiPr2}. Let $\widetilde{p}_0$ be an almost holomorphic continuation of the leading symbol of
$\widetilde{P}$, bounded together with all of its derivatives in a tubular \neigh{} of $\Lambda_{\Phi_0}$. To simplify the notation, we shall write here $p:=\widetilde{p}_0$. According to (\ref{eq4.0}), we know that
\begeq
\label{eq4.10.1}
\Re\left(\left(1-ic \frac{\delta \eps}{\abs{x}^2}\right) {p}\left(x,\frac{2}{i}\frac{\partial \Phi_{\delta,\eps}(x)}{\partial x}\right)\right) \geq \frac{\delta \eps}{\widetilde{C}},\quad \abs{x}\geq \sqrt{\eps}.
\endeq
Following Section 6 of~\cite{HiPr2}, we shall now switch to rescaled variables. Set
\begeq
\label{eq4.11}
x = \sqrt{\eps} \widetilde{x}.
\endeq
In the new coordinates, the IR-manifold $\Lambda_{\Phi_{\delta,\eps}}$ in (\ref{eq4.2}) becomes replaced by the manifold
\begeq
\label{eq4.12}
\Lambda_{\widetilde{\Phi}_{\delta,\eps}}=\Big\{\Big(\tilde{x},\frac{2}{i} \frac{\partial \widetilde{\Phi}_{\delta,\eps}(\widetilde{x})}{\partial
\widetilde{x}}\Big) : \tilde{x} \in \comp^n\Big\},
\endeq
with
$$\widetilde{\Phi}_{\delta,\eps}(\widetilde{x}) = \frac{1}{\eps} \Phi_{\delta,\eps}(\sqrt{\eps}\widetilde{x}).$$
We notice that $\nabla^2 \widetilde{\Phi}_{\delta,\eps}\in L^{\infty}(\comp^n)$ uniformly in $\eps\in (0,\eps_0]$, $\delta\in (0,\delta_0]$, and that
along $\Lambda_{\widetilde{\Phi}_{\delta,\eps}}$, we have
$$
\widetilde{\xi} = -\Im \widetilde{x} +{\cal O}(\delta).
$$

\medskip
\noindent
Let us consider the $\widetilde{h}$--pseudodifferential operator,
\begeq
\label{eq4.13}
P_{\eps}:= \frac{1}{\eps} p^w(x,hD_x) = \frac{1}{\eps}p^w \left(\sqrt{\eps}\left(\widetilde{x},\widetilde{h}D_{\widetilde{x}}\right)\right),
\quad \widetilde{h}=\frac{h}{\eps}=\frac{1}{A},
\endeq
with the Weyl symbol given by
\begeq
\label{eq4.14}
p_{\eps}(\widetilde{x},\widetilde{\xi})=\frac{1}{\eps} p\left(\sqrt{\eps}(\widetilde{x},\widetilde{\xi})\right).
\endeq

\medskip
\noindent
It follows from (\ref{eq4.10.1}) that along the manifold $\Lambda_{\widetilde{\Phi}_{\delta,\eps}}$, the symbol (\ref{eq4.14}) satisfies the
following estimate,
\begeq
\label{eq4.15}
\Re \left(\left(1-ic \frac{\delta}{\abs{\widetilde{x}^2}}\right)
{p}_{\eps}\left(\widetilde{x},\frac{2}{i}\frac{\partial \widetilde{\Phi}_{\delta,\eps}(\widetilde{x})}{\partial \widetilde{x}}\right)\right)
\geq \frac{\delta}{\widetilde{C}},
\endeq
in the region where $\abs{\widetilde{x}}\geq 1$.

\medskip
\noindent
Associated with the IR-manifold $\Lambda_{\widetilde{\Phi}_{\delta,\eps}}$ is the weighted space $H_{\widetilde{\Phi}_{\delta,\eps},
\widetilde{h}}(\comp^n)$, where we notice that
$$
\frac{\widetilde{\Phi}_{\delta,\eps}(\widetilde{x})}{\widetilde{h}} = \frac{\Phi_{\delta,\eps}(x)}{h}.
$$
The map $u(x)\mapsto \widetilde{u}(\widetilde{x})=\eps^{n/2} u(\sqrt{\eps}\widetilde{x})$ then takes the space $H_{\Phi_{\delta,\eps},h}(\comp^n)$
unitarily onto the space $H_{\widetilde{\Phi}_{\delta,\eps},\widetilde{h}}(\comp^n)$.

\bigskip
\noindent
Let now $\chi(\widetilde{x})\in C^{\infty}_b(\comp^n;[0,1])$ be such that $\chi=1$ for large $\abs{\widetilde{x}}$, and with ${\rm supp}\, \chi$ contained in the set where $\abs{\widetilde{x}}\geq 1$. Let us set
$$
m(\widetilde{x})=1-ic\frac{\delta}{\abs{\widetilde{x}}^2}.
$$
Assume also that the spectral parameter $z\in \comp$ satisfies $\abs{z}\leq C$, for some fixed $C>0$. An application of Proposition 3 of~\cite{HiPr2}, as in (6.15) in \cite{HiPr2}, shows that the scalar product
\begeq
\label{eq4.16}
\left(\chi m ({P}_{\eps}-\widetilde{h}z)\widetilde{u},\widetilde{u}\right)_{\widetilde{\Phi}_{\delta,\eps},\widetilde{h}}
\endeq
is equal to
$$
\int \chi(\widetilde{x})m(\widetilde{x}){p}_{\eps}
\left(\widetilde{x},\frac{2}{i}\frac{\partial \widetilde{\Phi}_{\delta,\eps}(\widetilde{x})}{\partial \widetilde{x}}\right)
\abs{\widetilde{u}(\widetilde{x})}^2
e^{-2\widetilde{\Phi}_{\delta,\eps}(\widetilde{x})/\widetilde{h}}\, L(d\widetilde{x}) +
{\cal O}(\widetilde{h})\norm{\widetilde{u}}^2_{\widetilde{\Phi}_{\delta,\eps},\tilde{h}}.
$$
Thus,
\begin{multline*}
\Re \left(\chi m ({P}_{\eps}-\widetilde{h}z)\widetilde{u},\widetilde{u}\right)_{\widetilde{\Phi}_{\delta,\eps},\widetilde{h}} \\
= \int \chi(\widetilde{x}) \Re \left(m(\widetilde{x}){p}_{\eps}
\left(\widetilde{x},\frac{2}{i}\frac{\partial \widetilde{\Phi}_{\delta,\eps}(\widetilde{x})}{\partial \widetilde{x}}\right)\right)
\abs{\widetilde{u}(\widetilde{x})}^2 e^{-2\widetilde{\Phi}_{\delta,\eps}(\widetilde{x})/\widetilde{h}}\, L(d\widetilde{x}) \\ +
{\cal O}(\widetilde{h})\norm{\widetilde{u}}^2_{\widetilde{\Phi}_{\delta,\eps},\widetilde{h}},
\end{multline*}
and using that (\ref{eq4.15}) holds near the support of $\chi$, we get, by an application of the Cauchy-Schwarz inequality,
\begin{multline*}
\int \chi(\widetilde{x}) \abs{\widetilde{u}(\widetilde{x})}^2
e^{-2\widetilde{\Phi}_{\delta,\eps}(\widetilde{x})/\widetilde{h}}\, L(d\widetilde{x}) \\
\leq {\cal O}(1)
\norm{\chi ({P}_{\eps}-\widetilde{h}z)\widetilde{u}}_{\widetilde{\Phi}_{\delta,\eps},\widetilde{h}}\,
\norm{\widetilde{u}}_{\widetilde{\Phi}_{\delta,\eps},\widetilde{h}} +
{\cal O}(\widetilde{h})\norm{\widetilde{u}}^2_{\widetilde{\Phi}_{\delta,\eps},\widetilde{h}}.
\end{multline*}
Coming back to the original variable $x=\sqrt{\eps}\widetilde{x}$ and using that
$$
\norm{\chi ({P}_{\eps}-\widetilde{h}z)\widetilde{u}}_{\widetilde{\Phi}_{\delta,\eps},\widetilde{h}}
= \frac{1}{\eps} \norm{\chi\left(\frac{\cdot}{\sqrt{\eps}}\right) \left(p^w(x,hD_x)-hz\right)u}_{\Phi_{\delta,\eps},h},
$$
we obtain that
\begin{multline*}
\eps \int \chi\left(\frac{x}{\sqrt{\eps}}\right) \abs{u(x)}^2
e^{-2\Phi_{\delta,\eps}(x)/h}\, L(d{x}) \\
\leq {\cal O}(1)
\norm{\chi\left(\frac{\cdot}{\sqrt{\eps}}\right)(p^w(x,hD_x)-h z)u}_{\Phi_{\delta,\eps},h}\,
\norm{u}_{\Phi_{\delta,\eps},h} +
{\cal O}(h)\norm{u}^2_{\Phi_{\delta,\eps},h}.
\end{multline*}
An application of (\ref{eq4.7}) then gives,
\begin{multline*}
\eps \int \chi\left(\frac{x}{\sqrt{\eps}}\right) \abs{u(x)}^2
e^{-2\Phi_{\delta,\eps}(x)/h}\, L(d{x})
\leq {\cal O}(1)\norm{v}_{\Phi_{\delta,\eps},h}\,\norm{u}_{\Phi_{\delta,\eps},h} \\
+ {\cal O}(1)\norm{\chi\left(\frac{\cdot}{\sqrt{\eps}}\right)R_{-,h}u_-}_{\Phi_{\delta,\eps},h}\,\norm{u}_{\Phi_{\delta,\eps},h}
+ {\cal O}(h)\norm{u}^2_{\Phi_{\delta,\eps},h}.
\end{multline*}
Using Proposition \ref{prop2.2}, together with (\ref{eq3.2.1}) and (\ref{eq4.3.1}), we easily see that
\begeq
\label{eq4.17}
\norm{\chi\left(\frac{\cdot}{\sqrt{\eps}}\right)R_{-,h}u_-}_{\Phi_{\delta,\eps},h} = {\cal O}\left(\left(\frac{h}{\eps}\right)^{\infty}\right)\abs{u_-}.
\endeq

\medskip
\noindent
Recalling that $\eps=Ah$, we obtain the following exterior estimate,
\begin{multline}
\label{eq4.18}
h \int \chi\left(\frac{x}{\sqrt{Ah}}\right) \abs{u(x)}^2
e^{-2\Phi_{\delta,\eps}(x)/h}\, L(d{x}) \leq {\cal O}(1) \norm{v}_{\Phi_{\delta,\eps},h}\,\norm{u}_{\Phi_{\delta,\eps},h} \\
+ {\cal O}(A^{-\infty}) \abs{u_-}\,\norm{u}_{\Phi_{\delta,\eps},h}
+ {\cal O}\left(\frac{h}{A}\right)\norm{u}^2_{\Phi_{\delta,\eps},h}.
\end{multline}
The estimates (\ref{eq4.10}) and (\ref{eq4.18}) are the main results established in this subsection.

\subsection{End of the proof of Theorem 1.1}

In this subsection, we shall glue together the estimates (\ref{eq4.10}) and (\ref{eq4.18}), in order to show the well-posedness of the global Grushin
problem (\ref{eq4.7}). Applying the exterior estimate (\ref{eq4.18}) to estimate the last term occurring in the right hand side of (\ref{eq4.10}) and
adding the estimates (\ref{eq4.10}) and (\ref{eq4.18}), we obtain that
\begin{multline}
h\norm{u}^2 + h^{-1}\abs{u_-}^2 \leq \frac{{\cal O}(1)}{h}\norm{v}^2 + {\cal O}(1) \norm{v}\, \norm{u} \\
+ {\cal O}(h) \abs{v_+}^2 + {\cal O}(A^{-\infty})\abs{u_-}\,\norm{u} + \left({\cal O}_A(h^2) + {\cal O}\left(\frac{h}{A}\right)\right)\norm{u}^2.
\end{multline}
Here we have also used that we arrange, as we may, that $\chi+\chi_0^2 \geq 1$ on $\comp^n$. Now
$$
{\cal O}(1) \norm{v}\, \norm{u} + {\cal O}(A^{-\infty})\abs{u_-}\,\norm{u} \leq \frac{{\cal O}(1)}{h}\norm{v}^2 +
{\cal O}(A^{-\infty})h^{-1}\abs{u_-}^2 + \frac{h}{2}\norm{u}^2,
$$
and it follows that
\begin{multline*}
\frac{h^2}{2} \norm{u}^2 + \abs{u_-}^2 \leq {\cal O}(1) \norm{v}^2 + {\cal O}(h^2)\abs{v_+}^2 \\
+ {\cal O}(A^{-\infty})\abs{u_-}^2 + \left({\cal O}_A(h^3) + {\cal O}\left(\frac{h^2}{A}\right)\right)\norm{u}^2.
\end{multline*}
Taking the parameter $A$ sufficiently large but fixed, and then restricting the attention to the interval $h\in (0,h_0]$, for some $h_0>0$ small
enough depending on $A$, we obtain that
\begeq
\label{eq5.0}
h\norm{u} + \abs{u_-} \leq {\cal O}(1)\norm{v} + {\cal O}(h)\abs{v_+}.
\endeq
Here the norms throughout are taken in the space $H_{\Phi_{\delta,\eps},h}(\comp^n)$, and according to (\ref{eq4.3.1}), the weight function $\Phi_{\delta,\eps}$ can be replaced by the standard quadratic weight $\Phi_0$, at the expense of an ${\cal O}(1)$--loss.
The Grushin operator
\begeq
\label{eq5.1}
\widetilde{{\cal P}}(z;h) = \left( \begin{array}{ccc}
\widetilde{P}-hz & R_{-,h} \\\
R_{+,h} & 0
\end{array} \right): H_{\Phi_{0},h}(\comp^n) \times \comp^{N_0} \rightarrow H_{\Phi_{0},h}(\comp^n) \times \comp^{N_0}
\endeq
is therefore injective. On the other hand, being a finite rank perturbation of the Fredholm operator
$$
\left( \begin{array}{ccc}
\widetilde{P}-hz & 0 \\\
0 & 0
\end{array} \right): H_{\Phi_0,h}(\comp^n) \times \comp^{N_0} \rightarrow
H_{\Phi_0,h}(\comp^n) \times \comp^{N_0},
$$
the operator in (\ref{eq5.1}) is also Fredholm, and furthermore, as observed in the introduction, the index is zero. It follows that the Grushin
operator in (\ref{eq5.1}) is invertible, so that the problem (\ref{eq4.7}) is well-posed, for $z\in {\rm neigh}(\lambda_0+p_1(0),\comp)$.

\medskip
\noindent
The inverse of the operator in (\ref{eq5.1}) is of the form
\begeq
\label{eq5.2}
\widetilde{{\cal E}}(z;h) = \left( \begin{array}{ccc}
E(z;h) & E_+(z;h) \\\
E_-(z;h) & E_{-+}(z;h)
\end{array} \right): H_{\Phi_{0},h}(\comp^n) \times \comp^{N_0} \rightarrow
H_{\Phi_{0},h}(\comp^n) \times \comp^{N_0},
\endeq
and it follows from (\ref{eq5.0}) that
$$
E(z;h) = {\cal O}\left(\frac{1}{h}\right): H_{\Phi_0,h}(\comp^n)\rightarrow H_{\Phi_0,h}(\comp^n),
$$
with
$$
E_+(z;h) = {\cal O}(1): \comp^{N_0}\rightarrow H_{\Phi_0,h}(\comp^n),\quad E_-(z;h)={\cal O}(1): H_{\Phi_0,h}(\comp^n)\rightarrow \comp^{N_0},
$$
and
$$
E_{-+}(z;h) = {\cal O}(h): \comp^{N_0}\rightarrow \comp^{N_0}.
$$
Furthermore, let us recall from~\cite{SjZwGrushin} that $hz$, with $z\in {\rm neigh}(\lambda_0+p_1(0),\comp)$, is an eigenvalue of $P$ precisely
when the determinant of $E_{-+}(z;h)$ vanishes.

\medskip
\noindent
In Section 11.5 of~\cite{HeSjSt}, the action of the Grushin operator in (\ref{eq5.1}), after applying
the inverse FBI--Bargmann transformation, was studied in detail, on spaces of functions of the form
$$
\left(a(x;h)e^{i\Phi(x)/h},u_-\right),
$$
where $a(x;h)$ is a symbol and the quadratic form $\Phi$ has been introduced in (\ref{eq2.23}). It was deduced there that $E_{-+}(z;h)$ has an asymptotic expansion in half--integer powers of $h$, with a certain additional structure. A complete asymptotic expansion for the determinant
of $E_{-+}(z;h)$ was subsequently obtained and it was shown that it is a classical symbol of order $0$, and complete asymptotic expansions for the
zeros of the determinant were obtained using Puiseux series. That discussion goes through without any changes in the present situation, and therefore, repeating the arguments of Section 11.5 of~\cite{HeSjSt} as they stand, we obtain that the eigenvalues of $h^{-1}P$
in a sufficiently small but fixed \neigh{} of $\lambda_0 + p_1(0)$ have complete asymptotic expansions in powers of $h^{1/N_0}$, of the form
$$
\lambda(h) = \lambda_0 + p_1(0) + c_1 h^{1/N_0} + c_2 h^{2/N_0}+\ldots.
$$
On the other hand, from the main result of~\cite{HiPr2} and the discussion in the introduction, we know that for all $h>0$ small enough, the spectrum of $P$ in the disc $D(0,Ch)$, is contained in the union of the regions
$$
D\left(h(\lambda_0+p_1(0)), \frac{h}{\widetilde{C}}\right),
$$
where $\lambda_0$ is an eigenvalue of $q(x,D_x)$ with $\abs{\lambda_0 + p_1(0)}<C$, and $\widetilde{C}>0$ is a sufficiently large constant.
The statement of Theorem \ref{theo1} follows and this completes the proof.

\section{Proof of Theorem~\ref{theo}}
\setcounter{equation}{0}

We shall begin this section by explaining that it is actually sufficient to establish Theorem~\ref{theo} in the special case when $m=1$. Indeed,
when assuming that Theorem~\ref{theo} has already been proved when $m=1$, we may consider an order function $m \geq 1$ as in (\ref{eq1.1})
such that $m \in S(m)$; and a symbol $P(x,\xi;h)$ satisfying the associated assumptions of Theorem~\ref{theo}. Then, one can choose a
symbol $\widetilde{p}_0\in S(1)$ with a non-negative real part $\textrm{Re } \widetilde{p}_0\geq 0$ which is elliptic near infinity in
the symbol class $S(1)$; and such that $\widetilde{p}_0=p_0$ on a large compact set containing $p_0^{-1}(0)$ where $p_0$ stands for the
principal symbol of $P(x,\xi;h)$. This is for instance the case when taking $\chi_0\in C^{\infty}_0(\real^{2n};[0,1])$ such that $\chi_0=1$
near $p_0^{-1}(0)$ and setting
$$\widetilde{p}_0=\chi_0 p_0 + (1-\chi_0).$$
Defining also the symbols
$$\widetilde{p}_j=\chi_0 p_j + (1-\chi_0) \in S(1),$$
when $j \geq 1$, we may choose $\chi\in C^{\infty}_0(\real^{2n},[0,1])$ such that $\chi=1$ near $p_0^{-1}(0)$ and $\chi_0=1$ near $\rm{supp}\, \chi$.
By setting $P=P^w(x,hD_x;h)$ and $\tilde{P}=\tilde{P}^w(x,hD_x;h)$, where
$$\tilde{P}(x,\xi;h) \sim \sum_{j=0}^{+\infty}{\tilde{p}_j(x,\xi)h^j},$$
in the symbol class $S(1)$;
and using $L^2$--norms throughout, we deduce from the semiclassical elliptic regularity that
\begin{align*}
& \ \qquad h^{\frac{2k_0}{2k_0+1}}|z|^{\frac{1}{2k_0+1}} \|u\| \\
\leq & \  h^{\frac{2k_0}{2k_0+1}}|z|^{\frac{1}{2k_0+1}} \|\chi^w(x,hD_x) u\|+h^{\frac{2k_0}{2k_0+1}}|z|^{\frac{1}{2k_0+1}}\|(1-\chi)^w(x,hD_x)u\| \\
\leq & \ h^{\frac{2k_0}{2k_0+1}}|z|^{\frac{1}{2k_0+1}}\|\chi^w(x,hD_x) u\|+\mathcal{O}(h^{\frac{2k_0}{2k_0+1}}|z|^{\frac{1}{2k_0+1}}) \|(P-z)u\|+\mathcal{O}(h^{\infty})\|u\|,
\end{align*}
when $|z| \leq C_0$, for $0< C_0 \ll 1$,
since the principal symbol $p_0$ of the operator $P$ is elliptic near the support of the function $1-\chi$. By using that Theorem~\ref{theo} is
valid when $m=1$, we may apply it to the operator $\widetilde{P}$ to get that if $z$ is as in Theorem~\ref{theo},
\begin{align}
\label{eq8.1}
h^{\frac{2k_0}{2k_0+1}}|z|^{\frac{1}{2k_0+1}} \|\chi^w(x,hD_x) u\| \leq & \ \mathcal{O}(1)\|(\widetilde{P}-z)\chi^w(x,hD_x) u\| \\ \nonumber
\leq & \  \mathcal{O}(1)\|(P-z)\chi^w(x,hD_x) u\|+\mathcal{O}(h^{\infty})\|u\|,
\end{align}
since $(\widetilde{P}-P)\chi^w(x,hD_x) = \mathcal{O}(h^{\infty})$ in $\mathcal{L}(L^2)$ when $h \rightarrow 0^+$.
We get that
\begin{multline}
\label{eq8.2}
h^{\frac{2k_0}{2k_0+1}}|z|^{\frac{1}{2k_0+1}} \|\chi^w(x,hD_x) u\| \leq \mathcal{O}(1)\|(P-z)u\|+\mathcal{O}(1)\|[P,\chi^w(x,hD_x)]u\|\\
+\mathcal{O}(h^{\infty})\|u\|.
\end{multline}
When estimating the commutator term in the right hand side of (\ref{eq8.2}), we take $\widetilde{\chi}\in C^{\infty}_0(\real^{2n},[0,1])$ such
that $\widetilde{\chi}=1$ near $p_0^{-1}(0)$ and $\chi=1$ near $\rm{supp}\, \widetilde{\chi}$. Then, by using that
$$[P,\chi^w(x,hD_x)]\widetilde{\chi}^w(x,hD_x)=\mathcal{O}(h^{\infty}),$$
in $\mathcal{L}(L^2)$, together with the fact that $p_0$ is elliptic near the support of $1-\widetilde{\chi}$, we get that
\begin{equation}
\label{eq8.3}
h^{\frac{2k_0}{2k_0+1}}|z|^{\frac{1}{2k_0+1}} \|\chi^w(x,hD_x) u\| \leq \mathcal{O}(1)\|(P-z)u\| + \mathcal{O}(h^{\infty})\|u\|,
\end{equation}
which in view of previous estimates completes the proof of the reduction to the case when $m=1$. In what follows, we shall therefore be concerned
exclusively with the case when $m=1$.

Consider $p_0$ a symbol in the class $S(1)$ independent of the semiclassical parameter with a non-negative real part
$$\textrm{Re }p_0 \geq 0.$$
We assume that all the hypothesis
(\ref{eq1.5}), (\ref{eq1.6}), (\ref{eq1.6.5}), (\ref{re1}) and (\ref{ning2}) are fulfilled; and study the operator $p_0^w(x,hD_x)$
defined by the $h$-Weyl quantization of the symbol $p_0(x,\xi)$ with the following choice for the normalization of the Weyl quantization
\begin{equation}\label{quant1}
p_0^w(x,hD_x)u(x)=\int_{{\bf R}^{2n}}e^{2\pi i(x-y).\xi}p_0\Big(\frac{x+y}{2},h\xi\Big)u(y)dyd\xi.
\end{equation}
This normalization differs from the one considered in (\ref{quant}); but, of course, it is completely equivalent, after a rescaling of the
semiclassical parameter, to prove Theorem~\ref{theo} with the normalizations (\ref{quant})  or (\ref{quant1}). This choice is for convenience only.
Writing
$$p_0(X_j+Y)=q_j(Y)+\mathcal{O}(Y^3),$$
when $Y \rightarrow 0$; where $q_j$ denotes the quadratic approximation which begins the Taylor expansion of the symbol $p_0$ at the
doubly characteristic point $X_j$ and recalling (\ref{ning1}) and (\ref{ning3}); we notice that the following intersections of kernels are zero
\begin{equation}
\label{msj2}
\Big(\bigcap_{l=0}^{k_0}\textrm{Ker}\big[\textrm{Re }F_j(\textrm{Im }F_j)^l \big]\Big) \cap \real^{2n}=\{0\},
\end{equation}
for any $1 \leq j \leq N$.
One can deduce, as in \cite{Karel11}, that these properties imply that the following sums of $k_0+1$ non-negative quadratic forms
\begin{equation}
\label{msj5}
\sum_{l=0}^{k_0}\textrm{Re }q_j\big((\textrm{Im }F_j)^lX\big),
\end{equation}
when $1 \leq j \leq N$, are all positive definite. Indeed, let $X_0 \in \real^{2n}$ be such that
$$\sum_{l=0}^{k_0}\textrm{Re }q_j\big((\textrm{Im }F_j)^lX_0\big)=0.$$
The non-negativity of the quadratic form $\textrm{Re }q_j$ implies that for all $l=0,...,k_0$,
\begin{equation}
\label{msj1}
\textrm{Re }q_j\big((\textrm{Im }F_j)^lX_0\big)=0.
\end{equation}
By denoting $\textrm{Re }q_j(X;Y)$ the polar form associated to $\textrm{Re }q_j$, we deduce from the Cauchy-Schwarz inequality, (\ref{10}) and
(\ref{msj1}) that for all $l=0,...,k_0$ and $Y \in \real^{2n}$,
\begin{align*}
\big|\textrm{Re }q_j\big(Y;(\textrm{Im }F_j)^lX_0\big)\big|^2= & \ \big|\sigma\big(Y,\textrm{Re }F_j(\textrm{Im }F_j)^lX_0\big)\big|^2 \\
\leq & \ \textrm{Re }q_j(Y)\textrm{Re }q_j\big((\textrm{Im }F_j)^lX_0\big)=0.
\end{align*}
It follows that for all $l=0,...,k_0$ and $Y \in \real^{2n}$,
$$\sigma\big(Y,\textrm{Re }F_j(\textrm{Im }F_j)^lX_0\big)=0,$$
which implies that for all $l=0,...,k_0$,
$$\textrm{Re }F_j(\textrm{Im }F_j)^lX_0=0,$$
since $\sigma$ is non-degenerate. We finally obtain from (\ref{msj2}) that $X_0=0$, which proves that the quadratic forms (\ref{msj5}) are
positive definite. We then deduce from Proposition~2.0.1 in \cite{Karel11} that there exist real-valued weight functions
\begin{equation}
\label{msj0}
g_j \in S\big(1,\langle X\rangle^{-\frac{2}{2k_0+1}}dX^2\big),
\end{equation}
and positive constants $c_{1,j}$ and $c_{2,j}$ such that for all $X \in \real^{2n}$,
\begin{equation}
\label{msj3}
\textrm{Re }q_j(X)+c_{1,j}H_{\textrm{Im}q_j}\ g_j(X)+1 \geq c_{2,j} \langle X \rangle^{\frac{2}{2k_0+1}} \geq c_{2,j} |X|^{\frac{2}{2k_0+1}},
\end{equation}
where $H_{\textrm{Im}q_j}$ denotes the Hamilton vector field of $\textrm{Im }q_j$.
We use here the usual notation $S\big(\tilde{m}_{h},\tilde{M}_{h}^{-2}dX^2\big)$, where $\tilde{m}_h$ and $\tilde{M}_h$ are positive functions depending on the semiclassical parameter $h$, to stand for the symbol class
\begin{multline*}
S\big(\tilde{m}_h,\tilde{M}_h^{-2}dX^2\big)=\Big\{ a_h\in C^{\infty}(\real^{2n},\comp): \forall \alpha \in \nat^{2n}, \exists C_{\alpha}>0, \\
\forall X \in \real^{2n}, \forall \ 0<h \leq 1, \  |\partial_X^{\alpha} a_h(X)|
\leq C_{\alpha} \tilde{m}_h(X) \tilde{M}_h(X)^{-|\alpha|}\Big\}.
\end{multline*}
Setting
\begin{equation}\label{msj4}
g_{j,h}(X)=g_j\Big(\frac{X}{\sqrt{h}}\Big),
\end{equation}
for $0 <h \leq 1$; it follows from (\ref{msj3}) and the homogeneity properties of the quadratic form $q_j$ that for all $X \in \real^{2n}$ and
$0 < h \leq 1$,
\begin{multline}
\label{msj10}
h \textrm{Re }q_j\Big(\frac{X}{\sqrt{h}}\Big) +
c_{1,j}h(H_{\textrm{Im}q_j}\ g_j)\Big(\frac{X}{\sqrt{h}}\Big)+h =\textrm{Re }q_j(X)\\ +c_{1,j}h  (H_{\textrm{Im}q_j}\ g_{j,h})(X) +h
\geq c_{2,j} h^{\frac{2k_0}{2k_0+1}}|X|^{\frac{2}{2k_0+1}},
\end{multline}
since
\begin{multline*}
(H_{\textrm{Im}q_j}\ g_{j,h})(X)=\Big\{\textrm{Im }q_j(X),g_j\Big(\frac{X}{\sqrt{h}}\Big)\Big\}=
\Big\{h \textrm{Im }q_j\Big(\frac{X}{\sqrt{h}}\Big),g_j\Big(\frac{X}{\sqrt{h}}\Big)\Big\}\\
=h \Big\{\textrm{Im }q_j\Big(\frac{X}{\sqrt{h}}\Big),g_j\Big(\frac{X}{\sqrt{h}}\Big)\Big\}
= \{\textrm{Im }q_j,g_j\}\Big(\frac{X}{\sqrt{h}}\Big)=(H_{\textrm{Im}q_j}\ g_j)\Big(\frac{X}{\sqrt{h}}\Big),
\end{multline*}
where $\{p,q\}$ stands for the Poisson bracket
$$\{p,q\}=\frac{\partial p}{\partial \xi}.\frac{\partial q}{\partial x}-\frac{\partial p}{\partial x}.\frac{\partial q}{\partial \xi}.$$
Since $p_0 \in S(1)$, it follows from (\ref{eq1.6.5}) that there exists $c_{3} \geq 1$ such that for all $1 \leq j \leq N$ and $X \in \real^{2n}$,
\begin{equation}
\label{msj-3.14}
|p_0(X)| \leq c_{3} |X-X_j|^2
\end{equation}
and
\begin{equation}
\label{msj-3}
|p_0(X)-z| \geq \frac{|z|}{2} \textrm{ when } |X-X_j|^2 \leq \frac{|z|}{2c_{3}}.
\end{equation}
Recalling the assumption (\ref{re1}), one can find a positive constant $c_4>0$ such that for all $1 \leq j \leq N$ and $|X| \leq c_4$,
\begin{equation}
\label{re2}
r_j(X) \in \Gamma,
\end{equation}
where $r_j$ are the symbols defined in (\ref{hel1}), and $\Gamma$ is a closed angular sector with vertex at 0 included in the right open half-plane
$$\Gamma \setminus \{0\} \subset \big\{z \in \comp: \textrm{Re }z > 0\big\}.$$
One may assume that
\begin{equation}\label{ash1}
0 <c_4<\inf_{\substack{j,k=1,...,N\\ j \neq k}}|X_j-X_k|.
\end{equation}
One can therefore find a positive constant $c_5$ such that
\begin{equation}\label{msj6}
\forall \ 1 \leq j \leq N,\,\, \forall |X| \leq c_4, \ |\textrm{Im }r_j(X)| \leq c_5 \textrm{Re }r_j(X).
\end{equation}
Let $\psi$ be a $C_0^{\infty}(\real^{2n},[0,1])$ function such that
\begin{equation}\label{msj6b1}
\psi(X)=1, \textrm{ when } |X| \leq \frac{c_4}{2}; \textrm{ and } \textrm{supp }\psi \subset \big\{X \in \real^{2n} :  |X| \leq c_4\big\}.
\end{equation}
Setting
\begin{equation}\label{msj6b2}
\tilde{r}_j=\psi r_j,
\end{equation}
and recalling the well-known inequality
\begin{equation}\label{msj-1.5}
|f'(x)|^2 \leq 2 f(x)\|f''\|_{L^{\infty}({\bf R})},
\end{equation}
fulfilled by any non-negative smooth function $f$ with a bounded second derivative, we deduce from (\ref{msj6}) and (\ref{msj6b1}) that there
exists a positive constant $c_6$ such that for all $1 \leq j \leq N$ and $X \in \real^{2n}$,
\begin{equation}\label{msj8}
|\nabla \textrm{Re }\tilde{r}_j(X)| \leq c_6 \sqrt{\textrm{Re }\tilde{r}_j(X)}
\end{equation}
and
$$|c_5 \nabla \textrm{Re }\tilde{r}_j(X)-\nabla \textrm{Im }\tilde{r}_j(X)| \leq c_6 \sqrt{c_5\textrm{Re }\tilde{r}_j(X)-\textrm{Im }\tilde{r}_j(X)}.$$
It follows that
\begin{multline}
\label{msj9}
|\nabla \textrm{Im }\tilde{r}_j(X)| \leq |c_5 \nabla \textrm{Re }\tilde{r}_j(X)-
\nabla \textrm{Im }\tilde{r}_j(X)|+c_5| \nabla \textrm{Re }\tilde{r}_j(X)| \\
\leq  c_6 \sqrt{c_5 \textrm{Re }\tilde{r}_j(X)-\textrm{Im }\tilde{r}_j(X)}+c_5c_6\sqrt{\textrm{Re }\tilde{r}_j(X)},
\end{multline}
for all $X \in \real^{2n}$. We deduce from (\ref{hel1}), (\ref{msj10}), (\ref{msj6b1}) and (\ref{msj6b2}) that for all
$|Y| \leq \frac{c_4}{2}$ and $0<h \leq 1$,
\begin{multline}
\label{msj11}
\textrm{Re }p_0(X_j+Y)+c_{1,j}h  \big\{\textrm{Im }p_0(X_j+Y),g_{j,h}(Y)\big\} -
\textrm{Re }\tilde{r}_j(Y)\\ -c_{1,j}h  (H_{\textrm{Im}\tilde{r}_j}\ g_{j,h})(Y) +   h
\geq c_{2,j} h^{\frac{2k_0}{2k_0+1}}|Y|^{\frac{2}{2k_0+1}}.
\end{multline}
Since from (\ref{msj0}), (\ref{msj4}) and (\ref{msj9}),
\begin{multline*}
h  |(H_{\textrm{Im}\tilde{r}_j}\ g_{j,h})(X)| \lesssim h |\nabla \textrm{Im }\tilde{r}_j(X)||\nabla g_{j,h}(X)| \lesssim
\sqrt{h} |\nabla \textrm{Im }\tilde{r}_j(X)|\\
\lesssim   c_6 \sqrt{h} \sqrt{c_5 \textrm{Re }\tilde{r}_j(X)-\textrm{Im }\tilde{r}_j(X)}+c_5c_6 \sqrt{h}\sqrt{\textrm{Re }\tilde{r}_j(X)},
\end{multline*}
it follows from (\ref{msj11}) that there exists a positive constant $c_7$ such that for all $|Y|  \leq \frac{c_4}{2}$ and $0< h \leq 1$,
\begin{multline}
\label{msj12}
\textrm{Re }p_0(X_j+Y)+c_{1,j}h \big\{\textrm{Im }p_0(X_j+Y),g_{j,h}(Y)\big\} +  2h + c_7 \textrm{Re }\tilde{r}_j(Y)\\
+c_7 \big(c_5 \textrm{Re }\tilde{r}_j(Y)-\textrm{Im }\tilde{r}_j(Y) \big)
\geq c_{2,j} h^{\frac{2k_0}{2k_0+1}}|Y|^{\frac{2}{2k_0+1}}.
\end{multline}
Notice from (\ref{kps1}), (\ref{hel1}), (\ref{msj6}), (\ref{msj6b1}) and (\ref{msj6b2}), that for all $|Y| \leq \frac{c_4}{2}$,
\begin{multline*}
c_7 \textrm{Re }\tilde{r}_j(Y) +c_7 \big(c_5 \textrm{Re }\tilde{r}_j(Y)-\textrm{Im }\tilde{r}_j(Y) \big) \leq (2c_5+1)c_7\textrm{Re }\tilde{r}_j(Y)\\
\leq (2c_5+1)c_7\big(\textrm{Re }q_j(Y)+\textrm{Re }\tilde{r}_j(Y)\big)=(2c_5+1)c_7 \textrm{Re }p_0(X_j+Y).
\end{multline*}
It follows that for all $|Y|  \leq \frac{c_4}{2}$ and $0< h \leq 1$,
\begin{multline}
\label{msj12.33}
\big(1+(2c_5+1)c_7\big)\textrm{Re }p_0(X_j+Y)+c_{1,j}h \big\{\textrm{Im }p_0(X_j+Y),g_{j,h}(Y)\big\} +  2h  \\
\geq c_{2,j} h^{\frac{2k_0}{2k_0+1}}|Y|^{\frac{2}{2k_0+1}}.
\end{multline}
Let $C_0 \geq 1$ be a fixed constant. Then, by introducing the real-valued weight function
\begin{equation}
\label{msj21b}
g_{h}(X)= \sum_{j=1}^N{c_{1,j}\psi\big(2(X-X_j)\big)g_{j,h}(X-X_j)},
\end{equation}
where $\psi$ is the function defined in (\ref{msj6b1}); and noticing from (\ref{msj0}) and (\ref{msj4}) that
$$h\Big(\sum_{j=1}^Nc_{1,j}g_{j,h}(X-X_j)\Big)H_{\textrm{Im}p_0}\big[\psi\big(2(X-X_j)\big)\big]=\mathcal{O}(h),$$
we deduce from (\ref{eq1.5}), (\ref{eq1.6}), (\ref{ash1}) and (\ref{msj12.33})
that there exist some positive constant $c_8$, $c_9$ and $h_{0}$ such that for all $X \in \real^{2n}$ and $0<h \leq h_0$,
\begin{equation}
\label{msj20.5}
\textrm{Re }p_0(X)+h (H_{\textrm{Im}p_0} \ g_{h})(X)  +  c_8 h  \geq  c_{9} h^{\frac{2k_0}{2k_0+1}}\min\big[C_0^{\frac{1}{2}},(4c_3)^{\frac{1}{2}}
\delta(X)\big]^{\frac{2}{2k_0+1}},
\end{equation}
where $\delta$ stands for the distance to the set $(\textrm{Re }p_0)^{-1}(0)$.
Since $\textrm{Re }p_0 \geq 0$, we may also assume according to (\ref{msj0}), (\ref{msj3}), (\ref{msj4}) and (\ref{msj21b}) that
\begin{equation}\label{msj-1}
 \sup_{X \in {\bf R}^{2n}} |g_{h}(X)| \leq \frac{3}{4\pi},
\end{equation}
for all $0<h \leq h_0$.
Let $z$ be in $\comp$ and $0<h \leq h_0$. We shall use a multiplier method inspired by the one used by F.~H\'erau, J.~Sj\"ostrand and C.~Stolk in
\cite{HeSjSt}. By using the Wick quantization whose definition and properties are recalled in Section~\ref{Wick}, one can write that
\begin{align}
\label{msj13.1}
& \ \textrm{Re}\big([p_0(\sqrt{h}X)-z]^{\textrm{Wick}}u,[2- g_h(\sqrt{h}X)]^{\textrm{Wick}}u\big) \\ \nonumber
=& \ \textrm{Re}\big([2- g_h(\sqrt{h}X)]^{\textrm{Wick}}[p_0(\sqrt{h}X)-z]^{\textrm{Wick}}u,u\big)\\ \nonumber
=& \ \big(\textrm{Re}\big([2- g_h(\sqrt{h}X)]^{\textrm{Wick}}[p_0(\sqrt{h}X)-z]^{\textrm{Wick}}\big)u,u\big).
\end{align}
since real Hamiltonians get quantized in the Wick quantization by formally selfadjoint operators on $L^2(\real^n)$.
Notice also from (\ref{msj0}), (\ref{msj4}) and (\ref{msj21b}) that
\begin{equation}\label{msj13.2}
g_h(\sqrt{h}X) \in S(1,dX^2),
\end{equation}
uniformly with respect to the parameter $0<h \leq h_0$.
We deduce from symbolic calculus in the Wick quantization (\ref{lay4}) that
\begin{align}
\label{ng1}
& \ \textrm{Re}\big([2- g_h(\sqrt{h}X)]^{\textrm{Wick}} [p_0(\sqrt{h}X)-z]^{\textrm{Wick}}\big) \\ \nonumber
= \Big[\big(& \ 2- g_h(\sqrt{h}X)\big)\big(\textrm{Re }p_0(\sqrt{h}X)-\textrm{Re }z\big)
+\frac{\sqrt{h}}{4\pi}\nabla\big(g_h(\sqrt{h}X)\big).(\nabla \textrm{Re }p_0)(\sqrt{h}X)\\ \nonumber  + & \
\frac{1}{4\pi} h(H_{\textrm{Im}p_0}\ g_h)(\sqrt{h}X)\Big]^{\textrm{Wick}}+S_h,
\end{align}
with $\|S_h\|_{\mathcal{L}(L^2)}=\mathcal{O}(h)$.
Since from (\ref{msj-1.5}) and (\ref{msj13.2}), we have
$$\Big|\frac{\sqrt{h}}{4\pi}\nabla\big(g_h(\sqrt{h}X)\big).(\nabla \textrm{Re }p_0)(\sqrt{h}X)\Big| \lesssim
\sqrt{h} \sqrt{\textrm{Re }p_0(\sqrt{h}X)} \leq  \textrm{Re }p_0(\sqrt{h}X)+ \mathcal{O}(h),$$
it follows from (\ref{msj-1}) that there exists a positive constant $c_{10}$ such that for all $X \in \real^{2n}$ and $0<h \leq h_0$,
\begin{multline*}
\big(2-g_h(\sqrt{h}X)\big)\big(\textrm{Re }p_0(\sqrt{h}X)-\textrm{Re }z\big)+\frac{\sqrt{h}}{4\pi}\nabla\big(g_h(\sqrt{h}X)\big).
(\nabla \textrm{Re }p_0)(\sqrt{h}X)\\ +\frac{1}{4\pi} h(H_{\textrm{Im}p_0}\ g_h)(\sqrt{h}X)
\geq  \\
\frac{1}{4\pi}\textrm{Re }p_0(\sqrt{h}X)+\frac{1}{4\pi} h(H_{\textrm{Im}p_0}\ g_h)(\sqrt{h}X)-c_{10} h - \frac{9}{4} \max(0,\textrm{Re }z).
\end{multline*}
It follows from (\ref{msj20.5}) that for all $X \in \real^{2n}$ and $0<h \leq h_0$,
\begin{align*}
 & \ \big(2-g_h(\sqrt{h}X)\big)\big(\textrm{Re }p_0(\sqrt{h}X)-\textrm{Re }z\big)+\frac{\sqrt{h}}{4\pi}\nabla\big(g_h(\sqrt{h}X)\big).
 (\nabla \textrm{Re }p_0)(\sqrt{h}X)\\
 & \ +\frac{1}{4\pi} h(H_{\textrm{Im}p_0}\ g_h)(\sqrt{h}X) \geq  - \Big(\frac{1}{4\pi}c_8+c_{10}\Big) h -\frac{9}{4} \max(0,\textrm{Re }z)\\
  & \
 + \frac{1}{4\pi}c_{9} h^{\frac{2k_0}{2k_0+1}}\min\big[C_0^{\frac{1}{2}},(4c_3)^{\frac{1}{2}}\delta(\sqrt{h}X)\big]^{\frac{2}{2k_0+1}}.
\end{align*}
We then obtain that for all $X \in \real^{2n}$ and $0<h \leq h_0$,
\begin{align*}
& \ \big(2-g_h(\sqrt{h}X)\big)\big(\textrm{Re }p_0(\sqrt{h}X)-\textrm{Re }z\big)+\frac{\sqrt{h}}{4\pi}\nabla\big(g_h(\sqrt{h}X)\big).
(\nabla \textrm{Re }p_0)(\sqrt{h}X)\\
& \ +\frac{1}{4\pi} h(H_{\textrm{Im}p_0}\ g_h)(\sqrt{h}X)  \geq   \frac{1}{8\pi} c_{9} h^{\frac{2k_0}{2k_0+1}}|z|^{\frac{1}{2k_0+1}}\\
& \ +\frac{1}{4\pi} c_{9} h^{\frac{2k_0}{2k_0+1}}
\Big(\min\big[C_0^{\frac{1}{2}},(4c_3)^{\frac{1}{2}}\delta(\sqrt{h}X)\big]^{\frac{2}{2k_0+1}}-|z|^{\frac{1}{2k_0+1}}\Big)\\
& \ +\frac{9}{4}\Big(\frac{1}{18\pi}c_{9} h^{\frac{2k_0}{2k_0+1}}|z|^{\frac{1}{2k_0+1}} -\max(0,\textrm{Re }z)\Big)-\Big(\frac{1}{4\pi}c_8+c_{10}\Big) h.
\end{align*}
Considering the set
\begin{equation}
\label{msj14}
\Omega_{C,h}=\Big\{z \in \comp: \textrm{Re }z \leq \frac{1}{18\pi}c_{9} h^{\frac{2k_0}{2k_0+1}}|z|^{\frac{1}{2k_0+1}}, \ Ch \leq |z| \leq C_0  \Big\},
\end{equation}
where $C \gg 1$ is a large constant whose value will be chosen later, and $\varphi \in C_0^{\infty}(\real,[0,1])$ such that
\begin{equation}\label{msj-4}
\varphi(X)=1 \textrm{ when } |X| \leq \frac{1}{4c_3}, \textrm{ and } \textrm{supp } \varphi \subset
\Big\{ X \in \real: |X| \leq \frac{1}{3c_3}\Big\},
\end{equation}
we notice that for all $X \in \real^{2n}$, $0<h \leq h_0$, $C \geq 1$ and $z \in \Omega_{C,h}$,
\begin{multline*}
 \frac{1}{4\pi} c_{9} h^{\frac{2k_0}{2k_0+1}}
\Big(\min\big[C_0^{\frac{1}{2}},(4c_3)^{\frac{1}{2}}\delta(\sqrt{h}X)\big]^{\frac{2}{2k_0+1}}-|z|^{\frac{1}{2k_0+1}}\Big)\\ +  \frac{9}{4}\Big(\frac{1}{18\pi}c_{9} h^{\frac{2k_0}{2k_0+1}}|z|^{\frac{1}{2k_0+1}} -\max(0,\textrm{Re }z)\Big)
 \geq  -\frac{1}{4\pi} c_{9} h^{\frac{2k_0}{2k_0+1}}|z|^{\frac{1}{2k_0+1}}\varphi\Big(\frac{\delta(\sqrt{h}X)^2}{|z|}\Big).
\end{multline*}
By noticing now that one can find a $C_0^{\infty}(\real^{2n},[0,1])$ function $\Phi$ such that
\begin{equation}\label{msj-4.11}
\Phi(X)=1 \textrm{ when } |X| \leq \frac{1}{\sqrt{4c_3}}, \textrm{ and } \textrm{supp } \Phi \subset  \Big\{ X \in \real^{2n} :
|X| \leq \frac{1}{\sqrt{3c_3}}\Big\};
\end{equation}
verifying for all $X \in \real^{2n}$, $0<h \leq h_0$ and $z \in \Omega_{C,h}$,
$$\varphi\Big(\frac{\delta(\sqrt{h}X)^2}{|z|}\Big) \leq \sum_{j=1}^N\Phi\Big(\frac{\sqrt{h}X-X_j}{\sqrt{|z|}}\Big),$$
we get that for all $X \in \real^{2n}$, $0<h \leq h_0$, $C \geq 1$ and $z \in \Omega_{C,h}$,
\begin{align*}
& \ \big(2-g_h(\sqrt{h}X)\big)\big(\textrm{Re }p_0(\sqrt{h}X)-\textrm{Re }z\big) \\
 + & \ \frac{\sqrt{h}}{4\pi}\nabla\big(g_h(\sqrt{h}X)\big).(\nabla \textrm{Re }p_0)(\sqrt{h}X) + \frac{1}{4\pi} h(H_{\textrm{Im}p_0}\ g_h)(\sqrt{h}X)  \\ \geq & \   \frac{1}{8\pi}c_{9}
 h^{\frac{2k_0}{2k_0+1}}|z|^{\frac{1}{2k_0+1}} -\Big(\frac{1}{4\pi}c_8+c_{10}\Big) h
  -\frac{1}{4\pi} c_{9} h^{\frac{2k_0}{2k_0+1}}|z|^{\frac{1}{2k_0+1}}\sum_{j=1}^N\Phi\Big(\frac{\sqrt{h}X-X_j}{\sqrt{|z|}}\Big).
\end{align*}
It follows from (\ref{msj13.1}), (\ref{ng1}) and (\ref{lay0.5}) that there exist some positive constants $c_{11}$ and $c_{12}$ such that for all
$0<h \leq h_0$, $C \geq 1$, $z \in \Omega_{C,h}$ and $u \in \mathcal{S}(\real^{n})$,
\begin{multline}
\label{msj17}
 \textrm{Re}\big([p_0(\sqrt{h}X)-z]^{\textrm{Wick}}u,[2- g_h(\sqrt{h}X)]^{\textrm{Wick}}u\big)+c_{11} h\|u\|_{L^2}^2 \\
  + c_{11} h^{\frac{2k_0}{2k_0+1}}|z|^{\frac{1}{2k_0+1}}\sum_{j=1}^N\Big(\Phi\Big(\frac{\sqrt{h}X-X_j}{\sqrt{|z|}}\Big)^{\textrm{Wick}}u,u\Big)
  \geq  c_{12} h^{\frac{2k_0}{2k_0+1}}|z|^{\frac{1}{2k_0+1}}\|u\|_{L^2}^2.
\end{multline}
Recalling (\ref{msj13.2}) and (\ref{msj14}), we deduce from the Cauchy-Schwarz inequality and (\ref{lay0}) that
there exist some positive constants $c_{13}$, $c_{14}$ and $c_{15}$ such that
for all $0<h \leq h_0$, $C \geq c_{13}$, $z \in \Omega_{C,h}$ and $u \in \mathcal{S}(\real^{n})$,
\begin{multline}
\label{msj18}
 c_{15} h^{\frac{2k_0}{2k_0+1}}|z|^{\frac{1}{2k_0+1}}\sum_{j=1}^N\Big\|\Phi\Big(\frac{\sqrt{h}X-X_j}{\sqrt{|z|}}\Big)^{\textrm{Wick}}u\Big\|_{L^2} \\
+ \big\|p_0(\sqrt{h}X)^{\textrm{Wick}}u-zu\|_{L^2} \geq  c_{14} h^{\frac{2k_0}{2k_0+1}}|z|^{\frac{1}{2k_0+1}}\|u\|_{L^2}.
\end{multline}
Since from (\ref{msj-4.11}), we have
\begin{equation}\label{msj19.1}
\Phi\Big(\frac{\sqrt{h}X-X_j}{\sqrt{|z|}}\Big) \in S\Big(1,\frac{|z|}{h}dX^2\Big),
\end{equation}
when $1 \leq j \leq N$, we notice from (\ref{msj14}), (\ref{lay1}) and (\ref{lay2}) that
$$
\Big\|\Phi\Big(\frac{\sqrt{h}X-X_j}{\sqrt{|z|}}\Big)^{\textrm{Wick}}u\Big\|_{L^2}=\Big\|\Phi\Big(\frac{\sqrt{h}X-X_j}{\sqrt{|z|}}\Big)^wu\Big\|_{L^2}+
\mathcal{O}\Big(\frac{1}{C}\Big)\|u\|_{L^2}
$$
and
$$\big\|p_0(\sqrt{h}X)^{\textrm{Wick}}u-zu\|_{L^2}=\big\|p_0(\sqrt{h}X)^wu-zu\|_{L^2}+\mathcal{O}(h)\|u\|_{L^2}.$$
We deduce from (\ref{msj18}) that there exist some positive constants $c_{16}$ and $c_{17}$ such that for all $0<h \leq h_0$, $C \geq c_{17}$,
$z \in \Omega_{C,h}$ and $u \in \mathcal{S}(\real^{n})$,
\begin{multline}
\label{msj18.12}
 c_{15} h^{\frac{2k_0}{2k_0+1}}|z|^{\frac{1}{2k_0+1}}\sum_{j=1}^N\Big\|\Phi\Big(\frac{\sqrt{h}X-X_j}{\sqrt{|z|}}\Big)^w u\Big\|_{L^2}
+ \big\|p_0(\sqrt{h}X)^w u-zu\|_{L^2}\\ \geq  c_{16} h^{\frac{2k_0}{2k_0+1}}|z|^{\frac{1}{2k_0+1}}\|u\|_{L^2}.
\end{multline}
We shall now study the quantity
$$\sum_{j=1}^N\Big\|\Phi\Big(\frac{\sqrt{h}X-X_j}{\sqrt{|z|}}\Big)^wu\Big\|_{L^2}.$$
To do so, we shall establish an a priori estimate similar to the one proved in~\cite{HeSjSt} (Proposition~4.1), namely that for all
$0<h \leq h_0$, $C \geq c_{17}$, $z \in \Omega_{C,h}$ and $u \in \mathcal{S}(\real^n)$,
\begin{equation}\label{msj20}
\frac{1}{|z|}\|p_0(\sqrt{h}X)^wu-zu\|_{L^2}+\mathcal{O}\Big(\sqrt{\frac{h}{|z|}}\Big)\|u\|_{L^2}
\gtrsim \sum_{j=1}^N\Big\|\Phi\Big(\frac{\sqrt{h}X-X_j}{\sqrt{|z|}}\Big)^wu\Big\|_{L^2}.
\end{equation}
In \cite{HeSjSt}, this estimate is proved on the FBI transform side. By using similar arguments, namely a second microlocalization,
we shall prove this estimate directly without any use of the FBI transform.

Let $\Psi$ be a $C_0^{\infty}(\real^{2n},[0,1])$ function such that
$$
\Psi=1 \textrm{ when } |Y| \leq \frac{1}{\sqrt{3c_3}}; \textrm{ and } \textrm{supp }
\Psi \subset \Big\{Y \in \real^{2n} : |Y| \leq \frac{1}{\sqrt{2c_3}}\Big\}.
$$
We notice from (\ref{msj-3.14}) and (\ref{msj-3}) that the symbols
$$
\frac{1}{|z|}p_0(\sqrt{|z|}Y+X_j)\Psi(Y),
$$
where $1 \leq j \leq N$, are uniformly bounded together with all their derivatives with respect to the parameter $z$ when $z$ belongs to
$\Omega_{C,h}$; and that these symbols are elliptic
$$
\Big|\frac{1}{|z|}p_0(\sqrt{|z|}Y+X_j)\Psi(Y)-\frac{z}{|z|}\Big| \geq \frac{1}{2},
$$
on the set
$$\Big\{Y \in \real^{2n} : |Y| \leq \frac{1}{\sqrt{3c_3}}\Big\}.$$
When quantizing these symbols in the $\tilde{h}$-Weyl quantization with the new semiclassical parameter
\begin{equation}\label{mel1}
\tilde{h}=\frac{h}{|z|},
\end{equation}
we deduce from (\ref{msj-4.11}) and this ellipticity property that for all $0<h \leq h_0$, $C \geq c_{17}$, $z \in \Omega_{C,h}$ and
$u \in \mathcal{S}(\real^n)$,
$$
\big\|\Phi(\sqrt{\tilde{h}}Y)^wu\big\|_{L^2} \leq
\mathcal{O}(1)\Big\|\frac{1}{|z|}p_0(\sqrt{|z|}\sqrt{\tilde{h}}Y+X_j)^wu -\frac{z}{|z|}u\Big\|_{L^2}+\mathcal{O}(\tilde{h})\|u\|_{L^2}.
$$
We recall from (\ref{msj14}) that
$$\tilde{h}=\frac{h}{|z|} \leq \frac{1}{C} \ll 1,$$
where the large constant $C \gg 1$ appearing in (\ref{msj14}) remains to be chosen.
One can then deduce from (\ref{mel1}) and the symplectic invariance property of the Weyl quantization (Theorem~18.5.9 in~\cite{Hormander})
while using the following affine symplectic transformation
$$
X \mapsto X-\frac{1}{\sqrt{h}}X_j,
$$
that for all $0<h \leq h_0$, $C \geq c_{18}$, $z \in \Omega_{C,h}$ and $u \in \mathcal{S}(\real^n)$,
$$
\Big\|\Phi\Big(\frac{\sqrt{h}X-X_j}{\sqrt{|z|}}\Big)^wu\Big\|_{L^2} \leq \mathcal{O}(1)\frac{1}{|z|}\|p_0(\sqrt{h}X)^wu -zu\|_{L^2}+
\mathcal{O}\Big(\frac{h}{|z|}\Big)\|u\|_{L^2},
$$
where $c_{18}$ is a large positive constant and $h_0$ a new positive constant with $0<h_0 \ll 1$.
This proves the estimate (\ref{msj20}). We can next conclude as follows.
Noticing from (\ref{msj14}) that
$$
\mathcal{O}\Big(\sqrt{\frac{h}{|z|}}\Big)=\mathcal{O}\Big(\frac{1}{\sqrt{C}}\Big)
$$
and
$$
h^{\frac{2k_0}{2k_0+1}}|z|^{\frac{1}{2k_0+1}}\mathcal{O}\Big(\frac{1}{|z|}\Big)=\mathcal{O}\Big(\frac{1}{C^{\frac{2k_0}{2k_0+1}}}\Big),
$$
we deduce from (\ref{msj18.12}) and (\ref{msj20}) that there exist some positive constants $c_0$ and $\tilde{c}_0$ such that for all
$0<h \leq h_0$, $C \geq c_0$, $z \in \Omega_{C,h}$ and $u \in \mathcal{S}(\real^n)$,
\begin{equation}
\label{msj23}
\big\|p_0(\sqrt{h}X)^wu-zu\|_{L^2} \geq \tilde{c}_0 h^{\frac{2k_0}{2k_0+1}}|z|^{\frac{1}{2k_0+1}}\|u\|_{L^2}.
\end{equation}
Finally, we deduce from the symplectic invariance property of the Weyl quantization (Theorem~18.5.9 in~\cite{Hormander}) while using the linear
symplectic transformation
$$
(x,\xi) \mapsto (h^{-\frac{1}{2}}x,h^{\frac{1}{2}}\xi),
$$
that we have
for all $0<h \leq h_0$, $C \geq c_0$, $z \in \Omega_{C,h}$ and $u \in \mathcal{S}(\real^n)$,
\begin{equation}
\label{msj23.4}
\big\|p_0^w(x,hD_x)u-zu\|_{L^2} \geq \tilde{c}_0 h^{\frac{2k_0}{2k_0+1}}|z|^{\frac{1}{2k_0+1}}\|u\|_{L^2}.
\end{equation}
Recalling that $m=1$ and noticing from the asymptotic expansion (\ref{xi1}) and the Calder\'on-Vaillancourt Theorem that
$$\|P^w(x,hD_x;h) - p_0^w(x,hD_x)\|_{\mathcal{L}(L^2)}=\mathcal{O}(h),$$
when $h \rightarrow 0$, we finally obtain by possibly increasing the value of the positive constant $c_0>0$ that for all
$0<h \leq h_0$, $C \geq c_0$, $z \in \Omega_{C,h}$ and $u \in \mathcal{S}(\real^n)$,
\begin{equation}
\label{msj23.5}
\big\|P^w(x,hD_x;h)u-zu\|_{L^2} \geq \frac{\tilde{c}_0}{2} h^{\frac{2k_0}{2k_0+1}}|z|^{\frac{1}{2k_0+1}}\|u\|_{L^2},
\end{equation}
where $h_0$ is a new positive constant such that $0<h_0 \ll 1$. This ends the proof of Theorem~\ref{theo}.

\begin{appendix}
\section{Appendix on Wick calculus}\label{Wick}
\setcounter{equation}{0}

The purpose of this section is to recall the definition and basic properties of the Wick quantization that we need for the proof of
Theorem~\ref{theo}. We follow here the presentation of the Wick quantization given by N.~Lerner in \cite{Lerner,cubo,birkhauser} and refer the reader to his
works for the proofs of the results recalled below.

The main property of the Wick quantization is its property of positivity, i.e., that non-negative Hamiltonians define non-negative operators
$$a \geq 0 \Rightarrow a^{\textrm{Wick}} \geq 0.$$ We recall that this is not the case for the Weyl quantization and refer to \cite{Lerner}
for an example of non-negative Hamiltonian defining an operator which is not non-negative.

Before defining properly the Wick quantization, we first need to recall the definition of the wave packets transform of a function
$u \in \mathcal{S}(\real^n)$,
$$Wu(y,\eta)=(u,\varphi_{y,\eta})_{L^2({\bf R}^n)}=2^{n/4}\int_{{\bf R}^n}{u(x)e^{- \pi (x-y)^2}e^{-2i \pi(x-y).\eta}dx}, \ (y,\eta) \in \real^{2n}.$$
where
$$\varphi_{y,\eta}(x)=2^{n/4}e^{- \pi (x-y)^2}e^{2i \pi (x-y).\eta}, \ x \in \mathbb{R}^n,$$
and $x^2=x_1^2+...+x_n^2$. With this definition, one can check (See Lemma 2.1 in \cite{Lerner}) that
the mapping $u \mapsto Wu$ is continuous from $\mathcal{S}(\real^n)$ to $\mathcal{S}(\real^{2n})$, isometric from $L^{2}(\real^n)$ to $L^2(\real^{2n})$
and that we have the
reconstruction formula
\begin{equation}\label{lay0.1}
\forall u \in \mathcal{S}(\real^n), \forall x \in \real^n, \ u(x)=\int_{{\bf R}^{2n}}{Wu(y,\eta)\varphi_{y,\eta}(x)dyd\eta}.
\end{equation}
We denote by $\Sigma_Y$ the operator defined in the Weyl quantization by the symbol
$$p_Y(X)=2^n e^{-2\pi|X-Y|^2}, \ Y=(y,\eta) \in \real^{2n},$$
by using the same normalization
\begin{equation}\label{lay3}
(a^wu)(x)=\int_{{\bf R}^{2n}}{e^{2i\pi(x-y).\xi}a\Big(\frac{x+y}{2},\xi\Big)u(y)dyd\xi},
\end{equation}
as in \cite{Lerner}.
This operator is a rank-one orthogonal projection
$$\big{(}\Sigma_Y u\big{)}(x)=Wu(Y)\varphi_Y(x)=(u,\varphi_Y)_{L^2({\bf R}^n)}\varphi_Y(x),$$
and we define the Wick quantization of any $L^{\infty}(\real^{2n})$  symbol $a$ as
\begin{equation}\label{lay0.2}
a^{\textrm{Wick}}=\int_{{\bf R}^{2n}}{a(Y)\Sigma_Y dY}.
\end{equation}
More generally, one can extend this definition when the symbol $a$ belongs to $\mathcal{S}'(\real^{2n})$ by defining the operator
$a^{\textrm{Wick}}$ for any $u$ and $v$ in $\mathcal{S}(\real^{n})$ by
$$<a^{\textrm{Wick}}u,\overline{v}>_{\mathcal{S}'({\bf R}^{n}),\mathcal{S}({\bf R}^{n})}=
<a(Y),(\Sigma_Yu,v)_{L^2({\bf R}^n)}>_{\mathcal{S}'({\bf R}^{2n}),\mathcal{S}({\bf R}^{2n})},$$
where $<\textrm{\textperiodcentered},\textrm{\textperiodcentered}>_{\mathcal{S}'({\bf R}^n),\mathcal{S}({\bf R}^n)}$ denotes the duality bracket
between the spaces $\mathcal{S}'(\real^n)$ and $\mathcal{S}(\real^n)$. The Wick quantization is a positive quantization
\begin{equation}\label{lay0.5}
a \geq 0 \Rightarrow a^{\textrm{Wick}} \geq 0.
\end{equation}
In particular, real Hamiltonians get quantized in this quantization by formally self-adjoint operators and one has
(See Proposition 3.2 in \cite{Lerner}) that $L^{\infty}(\real^{2n})$ symbols define bounded operators on $L^2(\real^n)$ such that
\begin{equation}
\label{lay0}
\|a^{\textrm{Wick}}\|_{\mathcal{L}(L^2({\bf R}^n))} \leq \|a\|_{L^{\infty}({\bf R}^{2n})}.
\end{equation}
According to Proposition~3.3 in~\cite{Lerner}, the Wick and Weyl quantizations of a symbol $a$ are linked by the following identities
\begin{equation}\label{lay1bis}
a^{\textrm{Wick}}=\tilde{a}^w,
\end{equation}
with
\begin{equation}\label{lay2bis}
\tilde{a}(X)=\int_{{\bf R}^{2n}}{a(X+Y)e^{-2\pi |Y|^2}2^ndY}, \ X \in \real^{2n},
\end{equation}
and
\begin{equation}\label{lay1}
a^{\textrm{Wick}}=a^w+r(a)^w,
\end{equation}
where $r(a)$ stands for the symbol
\begin{equation}\label{lay2}
r(a)(X)=\int_0^1\int_{{\bf R}^{2n}}{(1-\theta)a''(X+\theta Y)Y^2e^{-2\pi |Y|^2}2^ndYd\theta}, \ X \in \real^{2n}.
\end{equation}
We also recall the following composition formula obtained in the proof of Proposition~3.4 in~\cite{Lerner},
\begin{equation}\label{lay4}
a^{\textrm{Wick}} b^{\textrm{Wick}} =\Big{[}ab-\frac{1}{4 \pi} a'.b'+\frac{1}{4i \pi}\{a,b\} \Big{]}^{\textrm{Wick}}+S,
\end{equation}
with $\|S\|_{\mathcal{L}(L^2({\bf R}^n))} \leq d_n \|a\|_{L^{\infty}}\gamma_{2}(b),$
when $a \in L^{\infty}(\real^{2n})$ and $b$ is a smooth symbol satisfying
$$\gamma_2(b)=\sup_{X \in {\bf R}^{2n}, \atop T \in {\bf R}^{2n}, |T|=1}|b^{(2)}(X)T^2| < +\infty.$$
The term $d_n$ appearing in the previous estimate stands for a positive constant depending only on the dimension $n$; and the notation $\{a,b\}$
denotes the Poisson bracket
$$\{a,b\}=\frac{\partial a}{\partial \xi}.\frac{\partial b}{\partial x}-\frac{\partial a}{\partial x}.\frac{\partial b}{\partial \xi}.$$

\end{appendix}

\end{document}